%% file: potentialfunction.tex
\documentclass[1p]{elsarticle_modified}
\usepackage{abbrmath_seonhwa}
\usepackage{amsfonts}
\usepackage{amssymb}
\usepackage{amsmath}
\usepackage{amsthm}
\usepackage{amsbsy}
\usepackage{scalefnt}
\usepackage{caption,subcaption}
\usepackage[colorlinks]{hyperref}
\usepackage{color}
\usepackage{graphicx}
\usepackage{float}
\usepackage{setspace}
\usepackage{tikz-cd} 
\usepackage[numbers]{natbib}
\usepackage{setspace}
\theoremstyle{definition}
\newtheorem{thm}{Theorem}[section]
\newtheorem{prop}[thm]{Proposition}
\newtheorem{lemma}[thm]{Lemma}
\newtheorem{cor}[thm]{Corollary}
\newtheorem{defn}[thm]{Definition}
\newtheorem{exam}[thm]{Example}
\newtheorem{rmk}[thm]{Remark}
\theoremstyle{definition} 

\newenvironment{talign*}
{\let\displaystyle\textstyle\csname align*\endcsname}
{\endalign}

\def\sl{\textrm{SL}(2,\mathbb{C})}
\def\psl{\textrm{PSL}(2,\mathbb{C})}
\def\Zbb{\mathbb{Z}}

\def\Hb{{\overline{\mathbb{H}^3}}}

\begin{document}
	
	\begin{frontmatter}
	\title{On the potential functions for a link diagram}
	\author{Seokbeom Yoon}
	\address{Department of Mathematical Sciences, Seoul National University}

	\begin{abstract} For an oriented diagram of a link $L$ in the 3-sphere, Cho and 	 Murakami defined the potential function whose critical point, slightly different from the usual sense, corresponds to a boundary parabolic $\mathrm{PSL}(2,\mathbb{C})$-representation of $\pi_1(S^3 \setminus L)$. They also showed
	that the volume and Chern-Simons invariant of such a representation can be computed from the potential function with its partial derivatives. In this paper, we extend  the potential function to a $\psl$-representation that is not necessarily boundary parabolic.  Under a mild assumption, it leads us to a combinatorial formula for computing the volume and Chern-Simons invariant of a   $\mathrm{PSL}(2,\mathbb{C})$-representation of a closed 3-manifold.
	\end{abstract}

\end{frontmatter}


\section{Introduction}
Let $L$ be a link in the 3-sphere with a fixed diagram.  Motivated from the paper of Yokota \cite{yokota2002potential} regarding the volume conjecture, Cho and Murakami \cite{cho2013optimistic, cho_2016} defined the \emph{potential function} $W(w_1,\cdots,w_{n})$ satisfying the following properties: (i) a non-degenerate point $\mathbf{w} = (w_1,\cdots,w_{n}) \in (\Cbb^\times)^n=(\Cbb \setminus \{0\})^n$ satisfying 
	\begin{equation}\label{eqn:Critical}
	\textrm{exp} \left( w_j \dfrac{\partial W}{\partial w_j} \right) =1\quad \textrm{for all } 1 \leq j \leq n
	\end{equation} corresponds to a boundary parabolic representation $\rho_{\mathbf{w}} : \pi_1(S^3 \setminus L) \rightarrow \psl$
	(we shall clarify the meaning of a non-degenerate point in Section \ref{sec:PF}); 
	(ii) the volume and Chern-Simons invariant of $\rho_\mathbf{w}$ are given by \[ \sqrt{-1} (\textrm{Vol}(\rho_\mathbf{w})+ \sqrt{-1}\mkern1mu\textrm{CS}(\rho_\mathbf{w})) \equiv W_0(\mathbf{w}) \quad \textrm{mod } \pi^2 \Zbb\] where the function $W_0(w_1,\cdots,w_{n})$ is defined by
	\[ W_0:=W(w_1,\cdots,w_{n})-\displaystyle\sum_{j=1}^{n}  \mkern-2mu \left(w_j \dfrac{\partial W}{\partial w_j}\right)\textrm{log}\,w_j. \]
	Also, Cho \cite{cho2016optimistic} proved that (iii) any boundary representation $\rho :\pi_1(S^3 \setminus L)\rightarrow \psl$ which does not send a meridian of each component of $L$ to the identity matrix is detected by $W$, i.e. there exists a non-degenerate point $\mathbf{w}\in(\Cbb^\times)^n$ satisfying the equation (\ref{eqn:Critical}) such that the corresponding representation $\rho_\mathbf{w}$ agrees with $\rho$ up to conjugation.

Main aim of this paper is to extend the potential function to a representation that is not necessarily boundary parabolic. Precisely, we define a \emph{generalized potential function} \[\Wbb(\mathbf{w},\mathbf{m})=\Wbb(w_1,\cdots,w_{n},m_1,\cdots,m_h),\] where $h$ is the number of the components of $L$, and show that it satisfies analogous properties, Theorems \ref{thm:Main1}, \ref{thm:Main2} and \ref{thm:Main3}, to the potential function $W$.

\subsection{Main theorems}
We here give an overview of our main theorems. We enumerate the components of $L$ by $1 \leq i \leq h$ and let $\mu_i$ and $\lambda_i$ be a meridian and the canonical longitude of each component, respectively.
\begin{thm}\label{thm:Main1} A non-degenerate point $(\mathbf{w},\mathbf{m})\in(\Cbb^\times)^{n+h}$ satisfying
\begin{equation}\label{eqn:critical}
\textrm{exp} \left(w_j \dfrac{\partial \Wbb}{\partial w_j} \right) =1 \quad \textrm{for all } 1 \leq j \leq n
\end{equation} corresponds to a representation $\rho_{\mathbf{w},\mathbf{m}} : \pi_1(S^3 \setminus L)\rightarrow \psl$ such that the eigenvalues of $\rho_{\mathbf{w},\mathbf{m}}(\mu_i)$ are $m_i$ and $m_i^{-1}$ (up to sign) for all $1 \leq i \leq h$.\end{thm}

\begin{thm}\label{thm:Main2} Let $\rho:\pi_1(S^3 \setminus L)\rightarrow\psl$ be a representation such that $\rho(\mu_i) \neq \pm I$ for all $1 \leq i\leq h$. If $\rho$ admits a $\sl$-lifting, then there exists a non-degenerate point $(\mathbf{w},\mathbf{m})$ satisfying the equation (\ref{eqn:critical}) such that the corresponding representation $\rho_{\mathbf{w},\mathbf{m}}$ agrees with $\rho$ up to conjugation.	
\end{thm}

We remark that such a non-degenerate point $(\mathbf{w},\mathbf{m})$ can be explicitly constructed from a given representation $\rho$. See Examples \ref{ex:ex1} and \ref{ex:ex2}. 
We also stress that the assumption on $\sl$-lifting does not restrict too many cases. For instance, if $\textrm{tr}(\rho(\mu_i)) \neq 0$ for all $1\leq i \leq h$, then $\rho$ admits a lifting. In particular, any boundary parabolic representation has a lifting. Also, if $L$ is a knot, then any representation $\rho : \pi_1(S^3 \setminus L)\rightarrow \psl$ admits a lifting.

Let $M=S^3 \setminus \nu(L)$ be the link exterior where $\nu(L)$ denotes a tubular neighborhood of $L$.
For $\kappa=(\kappa_1,\cdots,\kappa_h) \in (\Qbb \cup \{\infty\} )^h$ we denote by $M_\kappa$ the manifold  obtained by Dehn filling along the slope $\kappa_i$ on each boundary torus of $M$. Here $\kappa_i=\infty$ means that we do not fill the corresponding boundary torus.

Let $\rho : \pi_1(M_\kappa) \rightarrow \psl$ be a representation.
If $M_\kappa$ has non-empty boundary, i.e. $\kappa_i=\infty$ for some $i$, we shall assume that $\rho$ is boundary parabolic so that the volume and Chern-Simons invariant of $\rho$ are well-defined. We refer  \cite{garoufalidis2015complex} for details. Regarding $\rho$ as a representation from $\pi_1(M)$ by compositing the inclusion $\pi_1(M)\rightarrow \pi_1(M_\kappa)$, we have
 \begin{equation} \label{eqn:df}
 \left\{
 \begin{array}{ll}
 \textrm{tr}(\rho(\mu_i))=\pm 2,\,\, \textrm{tr}(\rho(\lambda_i))=\pm 2  & \textrm{for } \kappa_i=\infty \\[3pt]
 \rho (\mu_i^{r_i}\lambda_i^{s_i})= \pm I	& \textrm{for } \frac{r_i}{s_i}=\kappa_i \neq \infty
 \end{array}
 \right.
 \end{equation} where $r_i$ and $s_i$ are coprime integers. 
 
 If we assume that $\rho: \pi_1(M)\rightarrow \psl$ admits a $\sl$-lifting and $\rho(\mu_i) \neq \pm I$ for all $1 \leq  i \leq h$, then by Theorems \ref{thm:Main1} and \ref{thm:Main2} there exists a non-degenerate point $(\mathbf{w},\mathbf{m})$ such that $\rho_{\mathbf{w},\mathbf{m}} = \rho$ up to conjugation where  $m_i$ is an eigenvalue of $\rho(\mu_i)$. It follows from the equation (\ref{eqn:df}) that for $\kappa_i \neq \infty$ we have $m_i^{r_i} l_i^{s_i}=\pm1$ and thus $r_i\textrm{log}\, m_i+s_i\textrm{log}\, l_i \equiv 0$ in modulo $\pi \sqrt{-1}$ where $l_i$ is an eigenvalue $\rho(\lambda_i)$. From coprimeness of the pair $(r_i,s_i)$, there are integers  $u_i$ and $v_i$ satisfying
 \[r_i \textrm{log}\,m_i + s_i \textrm{log}\,l_i + \pi \sqrt{-1}( r_i u_i + s_i v_i)=0.\]

\begin{thm}\label{thm:Main3} The volume and Chern-Simons invariant of $\rho : \pi_1(M_\kappa)\rightarrow \psl$ are given by \[\sqrt{-1}(\textrm{Vol}(\rho)+ \sqrt{-1}\mkern1mu\textrm{CS}(\rho)) \equiv \Wbb_0(\mathbf{w},\mathbf{m}) \quad \textrm{mod } \pi^2 \Zbb\] where the function $\Wbb_0(w_1,\cdots,w_{n},m_1,\cdots,m_h)$ is defined by
	\begin{equation*}
	\begin{array}{l}
	\Wbb_0:=\Wbb(w_1,\cdots,w_{n},m_1,\cdots,m_h)-\displaystyle\sum_{j=1}^{n}  \mkern-2mu \left(w_j \dfrac{\partial \Wbb}{\partial w_j}\right)\textrm{log}\,w_j\\[10pt]
	\quad \quad \quad  - \mkern -3mu \displaystyle\sum_{ \kappa_i\neq \infty} \mkern-3mu \left[\left(m_i \dfrac{\partial \Wbb}{\partial m_i}\right)(\textrm{log}\,m_i + u_i \pi \sqrt{-1}) -\dfrac{r_i}{s_i}(\textrm{log}\,m_i + u_i \pi \sqrt{-1})^2 \right].
	\end{array}
	\end{equation*}
\end{thm}

\subsection{Organization of the paper}
The paper is organized as follows. In Section \ref{sec:PF}, we give a definition of a generalized potential function and prove Theorem~\ref{thm:Main1}. 
We present main computation of the proof in Section \ref{sec:proof1}.
In Section \ref{sec:VCI}, we recall the notion of a deformed Ptolemy assignment \cite{yoon2018volume}, which is a main ingredient of this paper. We then prove Theorems \ref{thm:Main2} and \ref{thm:Main3} in Sections \ref{sec:proof2} and \ref{sec:VCSI}, respectively. We also present some examples that show how our theorems work in practical computation at the end of the paper.

\subsection{Acknowledgment} The author would like to thank Hyuk Kim for his guidance and encouragement.
 The author also would like to thank Jinseok Cho and Seonhwa Kim for their helpful discussion regarding Theorem \ref{thm:Main2}.
	
\section{Potential functions} \label{sec:PF}
Let $L$ be a link in $S^3$ with $h$ components.
Throughout the paper, we fix an oriented diagram, denoted also by $L$, of $L$.
We assume that every component of $L$ has at least one over-passing crossing and at least one under-passing crossing. This can be achieved by applying Reidemeister moves, if necessary.
We denote the number of the regions of $L$ by $n$.	

We assign a complex variable $w_j$ ($1\leq j \leq n$) to each region of $L$ and let $\mathbf{w}=(w_1,\cdots,w_{n})$. We also assign a complex  variable $m_i$ ($1 \leq i\leq h$) to each component of $L$ and let $\mathbf{m}=(m_1,\cdots,m_h)$.
For notational simplicity, we enumerate a region and a component of $L$ by the index of the variables assigned to them.
 For a crossing, say $c$, of $L$ we define
\allowdisplaybreaks
\begin{align*}
\Wbb_c(\mathbf{w},\mathbf{m}) &:= \textrm{Li}_2\left(\dfrac{w_m}{m_\beta w_j}\right)+\textrm{Li}_2\left(\dfrac{w_k}{m_\alpha w_j}\right)-\textrm{Li}_2\left(\dfrac{w_l}{m_\beta w_k}\right)-\textrm{Li}_2\left(\dfrac{w_l}{m_\alpha w_m}\right)\\
& \quad +\textrm{Li}_2\left(\dfrac{w_jw_l}{w_mw_k}\right)-\dfrac{\pi^2}{6}+\textrm{log}\left(\dfrac{w_m}{m_\beta w_j} \right)\textrm{log}\left(\dfrac{w_k}{m_\alpha w_j}\right)
\end{align*}
for Figure \ref{fig:crossing}(a) and 
\begin{align*}			
\Wbb_c(\mathbf{w},\mathbf{m}) &:=-\textrm{Li}_2\left(\dfrac{m_\beta w_m}{w_j}\right)-\textrm{Li}_2\left(\dfrac{m_\alpha w_k}{w_j}\right)+\textrm{Li}_2\left(\dfrac{m_\beta w_l}{w_k}\right)+\textrm{Li}_2\left(\dfrac{m_\alpha w_l}{w_m}\right)\\
&\quad -\textrm{Li}_2\left(\dfrac{w_jw_l}{w_mw_k}\right)+\dfrac{\pi^2}{6}-\textrm{log}\left(\dfrac{m_\beta w_m}{w_j} \right)\textrm{log}\left(\dfrac{m_\alpha w_k}{w_j}\right)
\end{align*} for Figure \ref{fig:crossing}(b). Recall that the dilogarithm function is given by $\textrm{Li}_2(z) = -\int_0^z \frac{\textrm{log} \, (1-t)}{t} dt$. See, for instance, \cite{zagier2007dilogarithm}. 
We then define the \emph{generalized potential function} \[\Wbb(\mathbf{w},\mathbf{m}) := \mkern -10mu\sum_{\textrm{crossing }c} \mkern -10mu\Wbb_{c}(\mathbf{w},\mathbf{m})\] where the sum is over all crossings of $L$.   Here and throughout the paper, we fix a branch of the logarithm; for actual computation we will use the principal branch having the imaginary part in the interval $(-\pi,\pi]$. 
\begin{figure}[!h]
	\centering
	\scalebox{1}{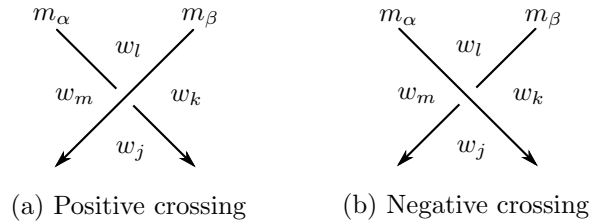} 
	\caption{Variables around a crossing}
	\label{fig:crossing}
\end{figure}
\begin{rmk} The generalized potential function $\Wbb$  reduces to the potential function $W$ in \cite{cho_2016,cho2016optimistic}  when $m_1=\cdots=m_h=1$.
\end{rmk}

\begin{defn} \label{defn:sol} (i) A point $(\mathbf{w},\mathbf{m})\in (\Cbb \setminus \{0\})^{n+h}=(\Cbb ^\times)^{n+h}$ is called a \emph{solution} if 
	\begin{equation} \label{eqn:diff}
	\textrm{exp} \left(w_j \dfrac{\partial \Wbb}{\partial w_j} \right) =1 \quad \textrm{for all } 1 \leq j \leq n.
	\end{equation}
	(ii) A point $(\mathbf{w},\mathbf{m})$ is said to be \emph{non-degenerated} if the following five values are not $1$ at each crossing of $L$:
	\begin{equation}\label{eqn:nond}
	\left\{
	\begin{array}{ll}
	\dfrac{w_m}{m_\beta w_j},\, \dfrac{w_k}{m_\alpha w_j},\, \dfrac{w_l}{m_\beta w_k},\, \dfrac{w_l}{m_\alpha w_m} ,\, \dfrac{w_j w_l}{w_m w_k} & \textrm{ for Figure \ref{fig:crossing}(a)} \\[15pt]
	\dfrac{m_\beta w_m}{w_j} ,\,\dfrac{m_\alpha w_k}{w_j},\, \dfrac{m_\beta w_l}{w_k} ,\, \dfrac{m_\alpha w_l}{w_m} ,\, \dfrac{w_j w_l}{w_k w_m}  &\textrm{ for Figure \ref{fig:crossing}(b)}.
	\end{array} 
	\right.
	\end{equation}	
\end{defn}

\begin{thm}[Theorem \ref{thm:Main1}]\label{thm:main1} A non-degenerate solution $(\mathbf{w},\mathbf{m})$ corresponds to a representation $\rho_{\mathbf{w},\mathbf{m}} : \pi_1(S^3 \setminus L)\rightarrow \psl$ such that the eigenvalues of $\rho_{\mathbf{w},\mathbf{m}}(\mu_i)$ are $m_i$ and $m_i^{-1}$ up to sign for all $1 \leq i \leq h$.  Here $\mu_i$ denotes a meridian of the $i$-th component of $L$.
\end{thm}

\subsection{Proof of Theorem \ref{thm:main1}} \label{sec:proof1}
Recall that we fixed a diagram of $L$ such that every component has at least one over-passing crossing and at least one under-passing crossing. Hence the space $S^3 \setminus (L \cup \{p,q\})$ decomposes into ideal octahedra (one per crossing) where $p \neq q \in S^3$ are two points not in $L$. We denote by $\Ocal$ this octahedral decomposition. It was introduced in \cite{thurston1999hyperbolic} and can be found in several articles, such as  \cite{yokota2002potential,weeks2005computation,cho2016optimistic,kim2016octahedral}. 
Following \cite{cho2016optimistic}, we subdivide each ideal octahedron into five ideal tetrahedra as in Figures \ref{fig:cross_ratio_pos} and \ref{fig:cross_ratio_neg} and denote by $\Tcal$ the resulting ideal triangulation of  $S^3 \setminus (L \cup \{p,q\})$. 

Recall that an ideal tetrahedron with mutually distinct vertices $z_0,z_1,z_2,z_3 \in \partial \Hb = \Cbb \cup \{\infty\}$ is determined up to isometry by the cross-ratio
\[z=[z_0:z_1:z_2:z_3]=\frac{(z_0-z_3)(z_1-z_2)}{(z_0-z_2)(z_1-z_3)} \in \Cbb \setminus \{0,1\},\]where the cross-ratio at each edge is given by one of $z, z'=\frac{1}{1-z}$, and $z''=1-\frac{1}{z}$ as in Figure~\ref{fig:cross}. 
\begin{figure}[!h] 
	\centering
	\scalebox{1}{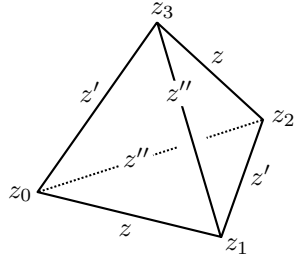} 
	\caption{Cross-ratios at the edges}
	\label{fig:cross}
\end{figure}
	
For a given non-degenerate solution $(\mathbf{w},\mathbf{m})=(w_1,\cdots,w_{n},m_1,\cdots,m_h)$, we assign the cross-ratio to each ideal tetrahedron of $\Tcal$ as in Figures \ref{fig:cross_ratio_pos} and \ref{fig:cross_ratio_neg}. 
The equation (\ref{eqn:nond}) guarantees that these tetrahedra are non-degenerated.
The product of the cross-ratios around each of edges that are created to divide the octahedra into tetrahedra is $1$ :
\begin{equation*}
\left\{
\begin{array}{ll}
\dfrac{w_m}{m_\beta w_j} \dfrac{m_\beta w_k}{w_l} \dfrac{w_j w_l}{w_m w_k} = 1 =  \dfrac{w_k}{m_\alpha w_j} \dfrac{m_\alpha w_m}{w_l} \dfrac{w_j w_l}{w_m w_k} & \textrm{ for Figure \ref{fig:crossing}(a)} \\[15pt]
\dfrac{m_\alpha w_l}{w_m} \dfrac{w_j}{m_\alpha w_k} \dfrac{w_k w_m}{w_j w_l} = 1 = \dfrac{w_j}{m_\beta w_m} \dfrac{m_\beta w_l}{w_k} \dfrac{w_k w_m}{w_j w_l} &\textrm{ for Figure \ref{fig:crossing}(b)}.
\end{array} 
\right.
\end{equation*}	 Therefore, at each crossing, five tetrahedra are well-glued to form an  octahedron. 
\begin{figure}[!h] 
	\centering
	\scalebox{1}{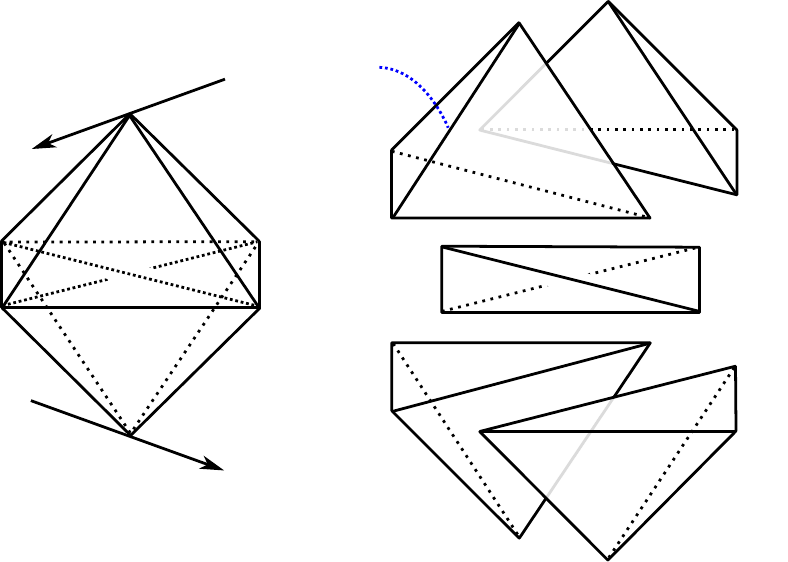} 
	\caption{Cross-ratios for Figure \ref{fig:crossing}(a)}
	\label{fig:cross_ratio_pos}
\end{figure}
\begin{figure}[!h] 
	\centering
	\scalebox{1}{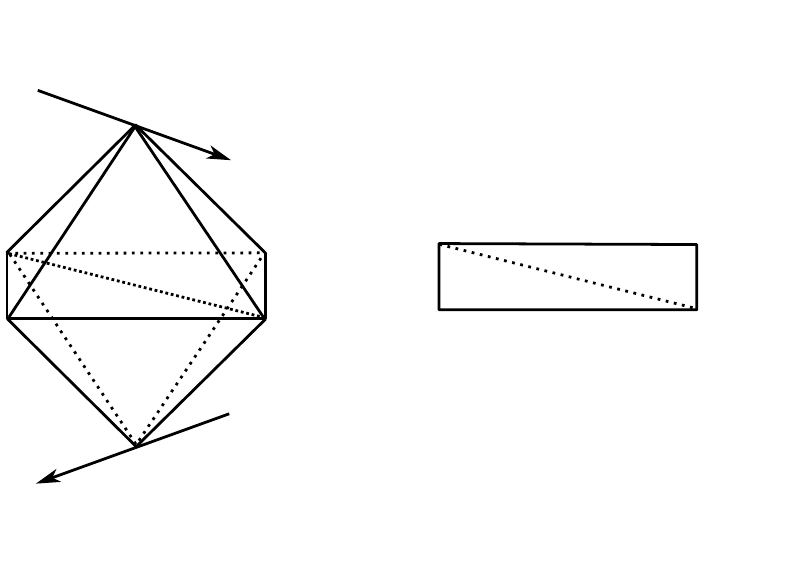} 
	\caption{Cross-ratios for Figure \ref{fig:crossing}(b)}
	\label{fig:cross_ratio_neg}
\end{figure}

We now check that cross-ratios given as in Figures \ref{fig:cross_ratio_pos} and \ref{fig:cross_ratio_neg} satisfy the gluing equations for the octahedral decomposition $\Ocal$, i.e. the product of the cross-ratios around each edge of $\Ocal$ is $1$. We thus shall obtain a representation \[\rho_{\mathbf{w},\mathbf{m}} : \pi_1(S^3 \setminus (L\cup\{p,q\})) = \pi_1(S^3 \setminus L) \rightarrow \psl\] up to conjugation as a holonomy representation.
Note that a similar computation can be found in \cite{cho2016optimistic} and \cite{kim2016octahedral}
 
Recall that $L$ has $n$ regions, so $n-2$ crossings. It thus has $n-2$ over-arcs and $n-2$ under-arcs. Here an over (resp., under)-arc is a maximal part of $L$ that does not under (resp., over)-pass a crossing. See Figure \ref{fig:local_diagram}. Recall also that the octahedral decomposition $\Ocal$ has $3n-4$ edges; (i)  $n$ \emph{regional edges} corresponding to the regions; (ii) $n-2$ \emph{over-edges} corresponding to the over-arcs;
(iii) $n-2$ \emph{under-edges} corresponding to the under-arcs.  We refer \cite[\S 3]{kim2016octahedral} for details.
\begin{figure}[!h] 
	\centering
	\scalebox{1}{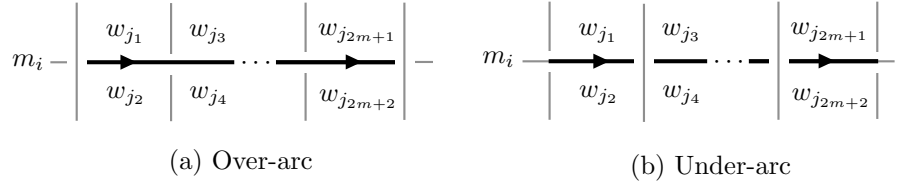} 
	\caption{Over- and under-arcs}
	\label{fig:local_diagram}
\end{figure}

Suppose an over-arc of $L$ over-passes $m$ crossings as in Figure \ref{fig:local_diagram}(a). Then around the corresponding over-edge, there are $4m+2$ cross-ratios; each of the over-passed crossings contributes 4 cross-ratios, and two crossings coming from the ends of the over-arc respectively contributes one cross-ratio (cf. Figure 10 in \cite{kim2016octahedral}). The product of these cross-ratios is 
\[\left(\dfrac{w_{j_1}}{m_i w_{j_2}}\right) \cdot \left(\dfrac{m_i w_{j_1}}{w_{j_2}}  \dfrac{w_{j_4}}{m_i w_{j_3}} \right)^{-1} \cdots \left(\dfrac{m_i w_{j_{2m-1}}}{w_{j_{2m}}}  \dfrac{w_{j_{2m+2}}}{m_i w_{j_{2m+1}}} \right)^{-1} \cdot \left(\dfrac{m_iw_{j_{2m+2}}}{ w_{j_{2m+1}}}\right)=1\]
for Figure \ref{fig:local_diagram}(a). Similarly, the product of cross-ratios around an under-edge is $1$ :
\[\left(\dfrac{w_{j_2}}{m_i w_{j_1}}\right) \cdot \left(\dfrac{m_i w_{j_2}}{w_{j_1}}  \dfrac{w_{j_3}}{m_i w_{j_4}} \right)^{-1} \cdots \left(\dfrac{m_i w_{j_{2m}}}{w_{j_{2m-1}}}  \dfrac{w_{j_{2m+1}}}{m_i w_{j_{2m+2}}} \right)^{-1} \cdot \left(\dfrac{m_i w_{j_{2m+1}}}{ w_{j_{2m+2}}}\right)=1\]
for Figure \ref{fig:local_diagram}(b).

Suppose a region of $L$ has $m$ crossings (or corners). The corresponding regional edge is represented by a horizontal edge of the octahedron at each of these crossings.
Therefore, there are $3m$ cross-ratios around the regional edge. See Figures \ref{fig:cross_ratio_pos} and \ref{fig:cross_ratio_neg} that three cross-ratios are attached to each horizontal edge. 
Let $\tau_{c,j}$ be the product of cross-ratios coming from a crossing $c$ and attached to the regional edge corresponding to the $j$-th region.
Then it is clear that the product of the cross-ratios around the regional edge corresponding to the $j$-th region is given by
\begin{equation}\label{eqn:tau}\prod_{\textrm{crossing } c} \mkern -10mu \tau_{c,j}\end{equation}
where the product is over all crossings appeared in the $j$-th region. On the other hand, $\tau$-values can be directly computed as follows from the cross-ratios given in Figures  \ref{fig:cross_ratio_pos} and~\ref{fig:cross_ratio_neg} :
 \begin{equation*} 
\left\{
\begin{array}{ll} 
\tau_{c,l}=\dfrac{(\frac{1}{m_\beta}w_l- w_k)(\frac{1}{m_\alpha}w_l- w_m)}{w_k w_m-w_j w_l}, & 		\tau_{c,k}=\dfrac{w_j w_l-w_k w_m}{( \frac{1}{m_\alpha}w_k- w_j)(m_\beta  w_k- w_l)}	 \\[12pt]	
\tau_{c,m}=\dfrac{w_j w_l-w_k w_m}{(\frac{1}{m_\beta}w_m-  w_j)(m_\alpha w_m- w_l)} , &
\tau_{c,j}=\dfrac{(m_\alpha w_j-w_k)( m_\beta  w_j-w_m)}{w_k w_m-w_j w_l}  \\[12pt]	 	 
\end{array} \right.
\end{equation*} for Figure \ref{fig:crossing}(a) and 
\begin{equation*}
\left\{
\begin{array}{ll} 
\tau_{c,l}= \dfrac{w_k w_m-w_j w_l}{(m_\beta w_l- w_k)(m_\alpha w_l-w_m)},& \tau_{c,k}=\dfrac{(m_\alpha w_k- w_j)(\frac{1}{m_\beta} w_k-w_l	)}{w_j w_l-w_k w_m}
\\[12pt]	
\tau_{c,m}=\dfrac{(m_\beta w_m- w_j)(\frac{1}{m_\alpha} w_m- w_l)}{w_j w_l-w_k w_m},& \tau_{c,j}= \dfrac{w_k w_m-w_j w_l}{(\frac{1}{m_\alpha} w_j-  w_k)(\frac{1}{m_\beta}w_j-w_m)}  \\[12pt]	 	 
\end{array} \right.
\end{equation*} for Figure \ref{fig:crossing}(b). Furthermore, a straightforward computation shows that \[ \tau_{c,j}=\textrm{exp} \left(w_j \frac{\partial \Wbb_c}{\partial w_j} \right)\] holds for any crossing $c$ and any region.
It thus follows from the equation (\ref{eqn:critical}) that the $\tau$-product in the equation (\ref{eqn:tau}) is 1, i.e. the product of the cross-ratios around each regional edges is~$1$.
\begin{rmk} Rewriting the equation (\ref{eqn:diff}) as the equation (\ref{eqn:tau}), one can checked that the equation (\ref{eqn:diff}) is invariant under change $m_i \mapsto \frac{1}{m_i}$ for all $1 \leq i\leq h$.	
\end{rmk}

We finally claim that the eigenvalues of $\rho_{\mathbf{w},\mathbf{m}}(\mu_i)$ are $m_i$ and $m_i^{-1}$. Since we assume that each component of $L$ has at least one over-passing crossing and at least one under-passing crossing, it contains a local diagram as in Figure \ref{fig:eigen}~(left). Then a meridian $\mu_i$ (up to base point) passes through two ideal tetrahedra coming from the ends as in Figure \ref{fig:eigen}~(middle). Therefore, the scaling factor of the holonomy action for $\mu_i$  is given by the product of two cross-ratios \[\left(\frac{w_j}{m_i w_k}\right)^{-1} \frac{m_iw_j}{w_k} = m_i^2.\] It follows that the eigenvalues of  $\rho_{\mathbf{w},\mathbf{m}}(\mu_i)  \in \psl$ are $m_i$ and $m_i^{-1}$ up to sign.
\begin{figure}[!h] 
	\centering
	\scalebox{1}{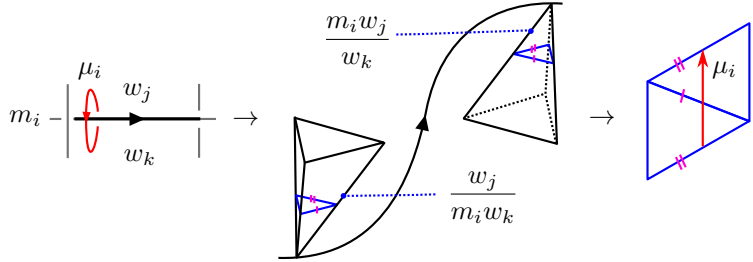} 
	\caption{A meridian }
	\label{fig:eigen}
\end{figure}

\section{Deformed Ptolemy assignments} \label{sec:VCI}

Let us briefly recall the notion of a deformed Ptolemy assignment \cite{yoon2018volume}, which is the key ingredient for proving Theorems \ref{thm:Main2} and \ref{thm:Main3}.

We let $\Tcal$ be an ideal triangulation of $S^3 \setminus (L \cup \{p,q\})$ given in Section \ref{sec:PF}. Replacing each ideal tetrahedron of $\Tcal$ by a truncated tetrahedron, we obtain a compact 3-manifold, say $N$, whose interior is homeomorphic to $S^3 \setminus (L \cup \{p,q\})$.
Here a \emph{truncated tetrahedron} is a polyhedron obtained from a tetrahedron by chopping off a small neighborhood of each vertex; see Figure \ref{fig:tetrahedron}.
Note that the boundary $\partial N$ is triangulated and is consisted of $h$ tori with two spheres.
We denote by $N^i$ and $\partial N^i$ the set of the oriented $i$-cells (unoriented when $i=0$) of $N$ and $\partial N$, respectively. We call an 1-cell of $\partial N$ a \emph{short edge} and call an 1-cell of $N$ not in $\partial N$ a \emph{long edge}. 

An assignment $\sigma : \partial N^1 \rightarrow \Cbb^\times$ is called a \emph{cocycle} if (i) $\sigma(e) \sigma(-e)=1$ for all $e \in \partial N^1$; (ii) $\sigma(e_1)\sigma(e_2)\sigma(e_3)=1$ whenever $e_1,e_2$, and $e_3$ bound, respecting an orientation, a 2-cell in $\partial N$. Here $-e$ denote the same 1-cell $e$ with its opposite orientation. A cocycle $\sigma : \partial N^1 \rightarrow \Cbb^\times$ induces a homomorphism $\pi_1(\Sigma) \rightarrow \Cbb^\times$ on each component $\Sigma$ of $\partial N$. For notational simplicity we denote all of such homomorphisms by $\overline{\sigma}$.

We denote by $\Tcal^1$ the set of the oriented $1$-cells of $\Tcal$ and identify each edge of $\Tcal$ with a long-edge of $N$ in a natural way (as in Figure \ref{fig:tetrahedron}). 
\begin{defn}[\cite{yoon2018volume}] \label{dfn:ptl}
For a given cocycle $\sigma : \partial N^1 \rightarrow \Cbb^\times$, an assignment $ c : \Tcal^1\rightarrow \Cbb^\times$ is called a \emph{$\sigma$-deformed Ptolemy assignment} if $c(-e)=-c(e)$ for all $e \in \Tcal^1$ and 
\begin{equation*} 
	c(l_3)c(l_6)=\dfrac{\sigma(s_{23})}{\sigma(s_{35})}\dfrac{\sigma(s_{26})}{\sigma(s_{65})}c(l_2)c(l_5)+\dfrac{\sigma(s_{13})}{\sigma(s_{34})}\dfrac{\sigma(s_{16})}{\sigma(s_{64})}c(l_1)c(l_4)
\end{equation*}
for each ideal tetrahedron $\Delta$ of $\Tcal$. Here $l_i$'s denote 1-cells of $\Delta$ and $s_{ij}$ denotes the 1-cell in $\partial N \cap \Delta$ running from $l_i$ to $l_j$ as in Figure \ref{fig:tetrahedron}. 
\begin{figure}[!h]
	\centering
	\scalebox{1}{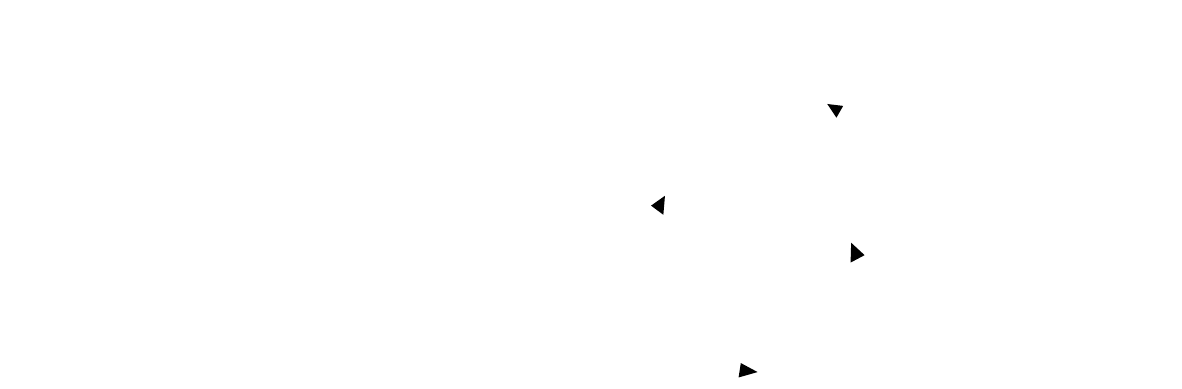}
	\caption{A truncated tetrahedron}
	\label{fig:tetrahedron}
\end{figure}
\end{defn}
It is proved in \cite{yoon2018volume} that a $\sigma$-deformed Ptolemy assignment $c$ corresponds to an assignment $\phi : N^1 \rightarrow \sl$ satisfying cocycle condition. It thus  corresponds to a representation $\rho_c : \pi_1(N)\rightarrow\sl$ up to conjugation. The cocycle $\phi$ can be explicitly given as follows:
  	\begin{equation*}
 	\phi(l_{j}) = \begin{pmatrix}
 	0 & -c(l_{j})^{-1} \\ c(l_{j}) & 0
 	\end{pmatrix}	
 	,\quad
 	\phi(s_{ij}) =\pm \begin{pmatrix}
 	\sigma(s_{ij}) & -\dfrac{\sigma(s_{ki})}{\sigma(s_{jk})} \dfrac{c(l_{k})}{c(l_{i}) c(l_{j})} \\[10pt] 0 & \sigma(s_{ij})^{-1}
 	\end{pmatrix}
 	\end{equation*} where the index $k$ is chosen so that $l_k$ and $s_{ij}$ lie on the same 2-cell. Also, $c$ determines the cross-ratio of each ideal tetrahedron of $\Tcal$. See Proposition 2.14 in \cite{yoon2018volume}. For instance, the cross-ratio at $l_3$ in Figure \ref{fig:tetrahedron} is given by 
	\[ 	\dfrac{\sigma(s_{12})\sigma(s_{45})}{\sigma(s_{24})\sigma(s_{51})}\,\dfrac{c(l_1) c(l_4)}{c(l_2)c(l_5)} \in \Cbb \setminus \{0,1\}.
	\]  We remark that these cross-ratios are non-degenerate (i.e. not $0,1,\infty$) and they satisfy the gluing equations for $\Tcal$ such that the holonomy representation coincides with $\rho_c$. We refer \cite{yoon2018volume} for details.
	
The following proposition shows how a $\sigma$-deformed Ptolemy assignment is related to the $\mathbf{w}$ and $\mathbf{m}$ in Section \ref{sec:PF}. Recall that $\Tcal$ has $n$ regional edges, each of which corresponds to a region of $L$. We orient these edges so that their initial points are the same (see Figures \ref{fig:octahedron} and \ref{fig:developing}), and denote them by $e_j$ ($1 \leq j \leq n$) according to the index of regions. Note that these edges appear as  horizontal edges of an octahedron as in Figure~\ref{fig:octahedron} (cf. Figure \ref{fig:crossing}).
\begin{prop} \label{prop:key} Let $\sigma : \partial N^1 \rightarrow \Cbb^\times$ be a cocycle that is trivial on the sphere components. Then for any $\sigma$-deformed Ptolemy assignment $c : \Tcal^1 \rightarrow \Cbb^\times$, \[(\mathbf{w},\mathbf{m})=\big(c(e_1),\cdots, c(e_n), \overline{\sigma}(\mu_1),\cdots,\overline{\sigma}(\mu_h)\big)\] is a non-degenerate solution such that $\rho_{\mathbf{w},\mathbf{m}}$ coincides with $\rho_c$, viewed as a $\psl$-representation, up to conjugation.
\end{prop}
\begin{proof} At each crossing of $L$, we denote edges of $\Tcal$ as in Figure \ref{fig:octahedron}. We orient these edges so that they coherent with the vertex-ordering given as in Figure \ref{fig:octahedron}. Recall that $h^2$ and $h^4$ are identified in $\Tcal$ and so are $h_2$ and $h_4$.
	 We denote by $s^{ij}$ (resp., $s_{ij}$) the short-edge running from $h^i$ to $h^j$ (resp., $h_i$ to $h_j$). For instance, $s^{42}$ and $s_{42}$ are short-edges winding the over-arc and under-arc, respectively.
	\begin{figure}[!h] 
		\centering
		\scalebox{1}{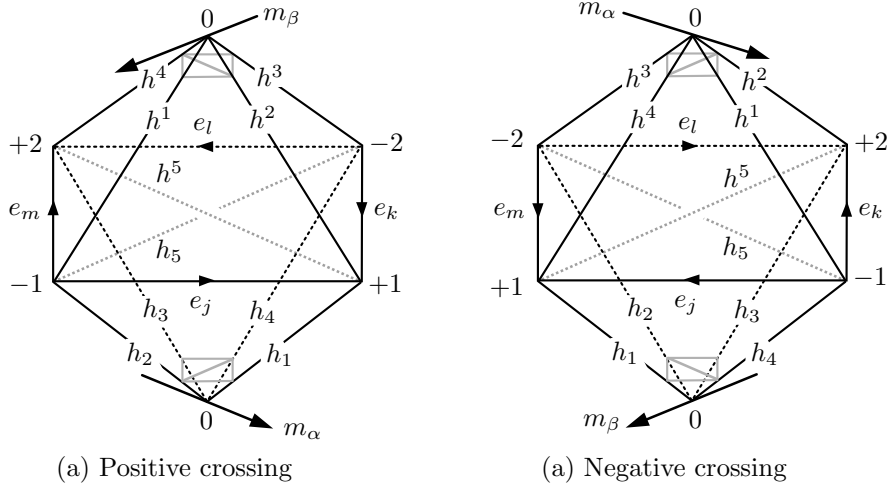} 
		\caption{Octahedron at a crossing.}
		\label{fig:octahedron}
	\end{figure}

	Applying Proposition 2.14 of \cite{yoon2018volume}, the cross-ratio at $h^1$ in Figure~\ref{fig:octahedron}(a) is given by \[\dfrac{c(h^2)c(e_m)}{\sigma(s^{42}) c(h^4)c(e_j)}=\dfrac{c(e_m)}{\sigma(s^{42})c(e_j)}=\dfrac{c(e_m)}{\overline{\sigma}(\mu_\beta)c(e_j)}.\] By the cross-ratio at $h^1$, we mean the cross-ratio at $l_3$ with respect to the tetrahedron chosen as in Figure \ref{fig:cross_ratio_pos}. We use terms the cross-ratios at $h^3,h^5, h_1,h_3$, in a same manner.
Similar computation gives us that the cross-ratios at $h^1,h^3,h^5,h_1,h_3$ for Figure~\ref{fig:octahedron}(a) are respectively given by	
	\[\dfrac{c(e_m)}{\overline{\sigma}(\mu_\beta)c(e_j)}, \frac{\overline{\sigma}(\mu_{\beta}) c(e_k)}{c(e_l)},\ \frac{c(e_j)c(e_l)}{c(e_m)c(e_k)}, \ \frac{c(e_k)}{\overline{\sigma}(\mu_{\alpha}) c(e_j)},\ \frac{\overline{\sigma}(\mu_{\alpha}) c(e_m)}{c(e_l)}\]
	and the cross-ratios at $h^1, h^3,h^5,h_1,h_3$ for Figure~\ref{fig:octahedron}(b) are respectively given by	
	\[ \frac{c(e_j)}{\overline{\sigma}(\mu_{\alpha}) c(e_k)},\ \frac{\overline{\sigma}(\mu_{\alpha}) c(e_l)}{c(e_m)},\ \frac{c(e_m)c(e_k)}{c(e_j)c(e_l)},\frac{c(e_j)}{\overline{\sigma}(\mu_{\beta}) c(e_m)},\ \frac{\overline{\sigma}(\mu_{\beta}) c(e_l)}{c(e_k)}.\]
	The proposition directly follows from comparing the above cross-ratios with the cross-ratios given in Figure \ref{fig:cross_ratio_pos} and \ref{fig:cross_ratio_neg}.
	We remark again that the above cross-ratios are non-degenerate and satisfy the gluing equations for $\Tcal$.
\end{proof}

For a representation $\rho:\pi_1(N)\rightarrow\sl$, we say that a cocycle $\sigma : \partial N^1 \rightarrow \Cbb^\times$ is \emph{associated to $\rho$} if \[\rho|_\Sigma(\gamma) = \begin{pmatrix} \overline{\sigma}(\gamma) & * \\ 0 & \overline{\sigma}(\gamma)^{-1} \end{pmatrix}\] up to conjugation for all $\gamma \in \pi_1(\Sigma)$ and for any component $\Sigma$ of $\partial N$. Here $\rho|_\Sigma: \pi_1(\Sigma) \rightarrow \sl$ means the restriction. Since every component $\Sigma$ of $\partial N$ is either a sphere or a torus, the restriction $\rho|_\Sigma$ is reducible. Therefore, for any representation $\rho$ there exists a cocycle $\sigma$ associated to $\rho$.

\begin{thm}\label{thm:main2} Let $\rho:\pi_1(N)\rightarrow\sl$ be a representation such that $\rho(\mu_i) \neq \pm I$ for all $1 \leq i\leq h$.  Then for any cocycle  $\sigma : \partial N^1 \rightarrow \Cbb^\times$ associated to $\rho$,  there exists a $\sigma$-deformed Ptolemy assignment $c$ such that $\rho_c=\rho$ up to conjugation.	
\end{thm}

A proof of Theorem \ref{thm:main2} is essentially also 	given in \cite[\S 4]{CYZ2018hikami} (see also \cite{cho2016optimistic}). The proof given in \cite{CYZ2018hikami} assume that $\rho$ is a (lifting of) boundary parabolic representation, but this is not actually required in the proof. For completeness of the paper, we present a detailed proof of Theorem \ref{thm:main2} in Section \ref{sec:proof2}.

\begin{cor}[Theorem \ref{thm:Main2}] \label{cor:main2}
	Let $\rho:\pi_1(N)\rightarrow\psl$ be a representation satisfying $\rho(\mu_i) \neq \pm I$ for all $1 \leq i\leq h$. If the representation $\rho$ admits a $\sl$-lifting, then there exists a non-degenerate solution $(\mathbf{w},\mathbf{m})$ such that $\rho_{\mathbf{w},\mathbf{m}}=\rho$ up to conjugation.	
\end{cor}
 \begin{proof} For each sphere component $\Sigma$ of $\partial N$, the restriction $\rho|_\Sigma :\pi_1(\Sigma)\rightarrow \sl$ is clearly trivial. Thus one can choose an associated cocycle $\sigma$ such that it is trivial on the sphere components. Then the proof directly follows from Proposition \ref{prop:key} and Theorem \ref{thm:main2}.
 \end{proof}
 \subsection{Proof of Theorem \ref{thm:main2}} \label{sec:proof2}
 For simplicity we may assume that a given cocycle  $\sigma : \partial N^1 \rightarrow \Cbb^\times$ is trivial on the sphere components. 
 Let $\widetilde{N}$ be the universal cover of $N$. 
 We lift $\sigma$ to $\partial\widetilde{N}$, and denote the resulting  cocycle also by $\sigma : \partial \widetilde{N}^1 \rightarrow \Cbb^\times$.
 	
 	\begin{defn} (\cite{yoon2018volume}) A \emph{decoration} $\Dcal : \widetilde{N}^0 \rightarrow \Cbb^2 \setminus \{ (0,0)^t\}$ is an assignment satisfying \begin{itemize}
 			\item  ($\rho$-equivariance) $\Dcal(\gamma \cdot v)=\rho(\gamma) \Dcal(v)$ for all $\gamma \in \pi_1(N)$ and $v \in \widetilde{N}^0$;
 			\item 	$\Dcal(v_2)= \sigma(s)\Dcal(v_1)$ for all $s \in \partial \widetilde{N}^1$ where $v_1$ and $v_2$ are the initial and terminal vertices of $s$, respectively.
 		\end{itemize} 
	\end{defn}
	We remark that a decoration exists, since a given cocycle $\sigma$ is associated to $\rho$. 
 	For a decoration $\Dcal$ we define $c : \Tcal^1 \rightarrow \Cbb$ by \[c(e) = \textrm{det}(\Dcal(v_1),\Dcal(v_2))\] for $e \in \Tcal^1$ where $v_1$ and $v_2$ are the initial and terminal vertices of any lifting of $e$, viewed as a long edge of $N$, respectively. Note that $c(e)$ does not depend on the choice of a lifting of $e$, since $\Dcal$ is $\rho$-equivariant. Also, note that $c(-e)=-c(e)$ for all $e \in \Tcal^1$.
 	\begin{prop} If $c(e) \neq 0 $ for all $e \in \Tcal^1$, then $c :\Tcal^1 \rightarrow \Cbb^\times$ is a $\sigma$-deforemd Ptolemy assignment.
 	\end{prop}
 	\begin{proof} Let us choose a lifting of an ideal triangulation $\Delta$ of $\Tcal$. We denote the edges of its truncation as in Definition \ref{dfn:ptl}; $l_i$ denotes a long-edge and $s_{ij}$ denotes the short edge running from $l_i$ to $l_j$. We also denote the initial and terminal vertices of $l_{i}$ by $v_{i}$ and $v^{i}$, respectively as in Figure \ref{fig:tetra}.	
 		\begin{figure}[!h]
 			\centering
 			\scalebox{1}{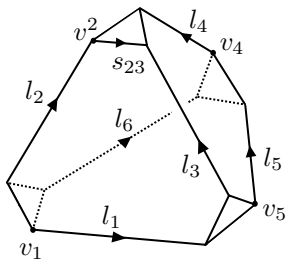}
 			\caption{A truncated tetrahedron.}
 			\label{fig:tetra}
 		\end{figure}
 		
 		Applying the Plucker relation to $\Dcal(v_{1}),\Dcal(v_{5}),\Dcal(v_{4}),\Dcal(v^{2})$, we obtain
 		\begin{talign*}
 			&\textrm{det}(\Dcal(v_{1}),\Dcal(v_{4}))\,\textrm{det}(\Dcal(v_{5}),\Dcal(v^{2})) \\
 			&= \textrm{det}(\Dcal(v_{1}),\Dcal(v_{5}))\, \textrm{det}(\Dcal(v_{4}),\Dcal(v^{2}))+\textrm{det}(\Dcal(v_{1}),\Dcal(v^{2}))\, \textrm{det}(\Dcal(v_{5}),\Dcal(v_{4})).
 		\end{talign*}
 		By construction of $c$, it is equivalent to
 		\begin{talign*}
 			 \sigma(s_{61}) \sigma(s_{64}) c(l_{6}) \,\sigma(s_{32})\sigma(s_{35}) c(l_{3}) 
 			& = \sigma(s_{15}) c(l_{1}) \sigma(s_{42}) c(l_{4})+ \sigma(s_{21}) c(l_{2}) \sigma (s_{54}) c(l_{5})\\
 			\Leftrightarrow\ c(l_3)c(l_6) 			&=\dfrac{\sigma(s_{23})}{\sigma(s_{35})}\dfrac{\sigma(s_{26})}{\sigma(s_{65})}c(l_2)c(l_5)+\dfrac{\sigma(s_{13})}{\sigma(s_{34})}\dfrac{\sigma(s_{16})}{\sigma(s_{64})}c(l_1)c(l_4).
 		\end{talign*}
 	Therefore, $c :\Tcal^1\rightarrow \Cbb^\times$ is a $\sigma$-deformed Ptolemy assignment.
 	\end{proof}
	Therefore, it is enough to claim that there exists a decoration  $\Dcal$ such that the induced assignment $c:\Tcal^1\rightarrow \Cbb$ satisfies $c(e) \neq 0 $ for all $e \in \Tcal^1$. 
 	
 	We first consider the regional edges $e_1,\cdots,e_n$ of $\Tcal$. We choose a lifting, $\widetilde{e}_j$, of each $e_j$ so that their terminal point agree as in Figure \ref{fig:developing}. Let $v^0_k$ and $v^1_k$ be the initial and terminal points of $\widetilde{e}_j$, viewed as an edge of $\widetilde{N}$, respectively. Since $\sigma : \partial N^1 \rightarrow \Cbb^\times$ is trivial on the sphere components, we have $\Dcal(v^1_j)= \Dcal(v^1_k)$. Moreover, from $\rho$-equivariance of $\Dcal$, we have \begin{equation}\label{eqn:equiv}
 	\Dcal(v^0_{j})=\rho(g) \Dcal(v^0_k)
 	\end{equation}
 	for some $g\in \pi_1(N)$.  From elementary covering theory one can check that if $e_j \cup e_k$ wraps an arc of $K$, then the loop $g$ should be the Wirtinger generator corresponding to the arc; see Figure \ref{fig:developing}. For simplicity we let $W=\Dcal(v^1_j)(=\Dcal(v^1_k))$ and $V_j=\Dcal(v^0_j)$ for $1 \leq j \leq m$. Note that $c(e_j) \neq 0 $ if and only if $\textrm{det}(W,V_j) \neq 0$.

 	We then consider the edges of $\Tcal$ that intersect $\nu(L)$. Let us consider an ideal triangle (with edges denoted by $x,y,e_k$) in $S^3 \setminus (L\cup \{p,q\})$ together with its lifting (with edges denoted by $\widetilde{x},\widetilde{y},\widetilde{e}_k$) as in Figure \ref{fig:developing}. Let $v_x$ and $v_y$ be the initial vertices of $\widetilde{x}$ and $\widetilde{y}$, again viewed as edges of $\widetilde{N}$, respectively. Then for the Wirtinger generator $g$, we have
 	\[ \rho(g) \Dcal(v_x)=\Dcal(g \cdot v_x)= \overline{\sigma}(g)^{\pm1} \Dcal(v_x). \]
 	Therefore, $\Dcal(v_x)$ is an eigenvector of $\rho(g)$.
 	It follows that $c(x) =\textrm{det}(W, \Dcal(v_x))\neq 0$  if and only if $W$ is not an eigenvector of $\rho(g)$. Similarly, $c(y) \neq 0 $ if and only if $V_k$ is not an eigenvector of $\rho(g)$.
 	\begin{figure}[!h]
 		\centering
 		\scalebox{1}{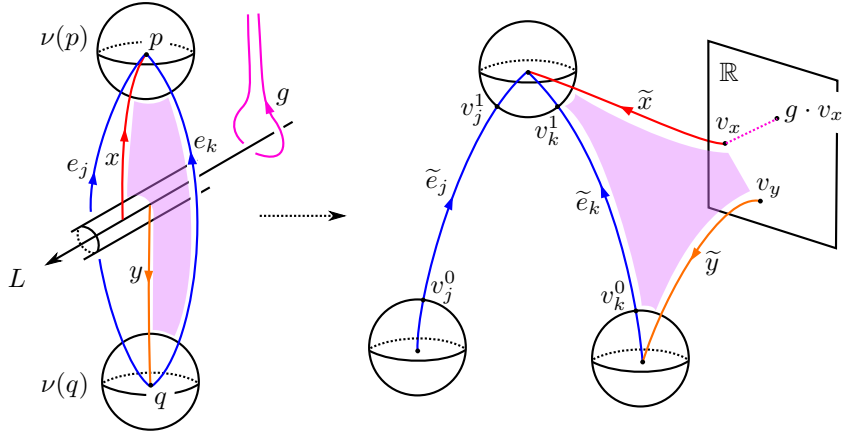} 
 		\caption{Local configuration of a lifting.}
 		\label{fig:developing}
 	\end{figure}

 	We finally consider an edge of $\Tcal$ that joins $q$ to itself.
 	Let us consider an ideal triangle (with edges denoted by $e_j,e_k,z$) in $S^3 \setminus (L \cup \{p,q\})$ together with its lifting (with edges denoted by $\widetilde{e}_j,\widetilde{e}_k,\widetilde{z}$) as in Figure \ref{fig:developing2}. It  follows that $c(z) \neq 0$ if and only if $\textrm{det}(V_j,V_k)=\textrm{det}(\rho(g)V_k,V_k) \neq0$ (recall the equation (\ref{eqn:equiv})). It is equivalent to the condition that $V_k$ is not an eigenvector of $\rho(g)$. Similarly, for an edge $z$ of $\Tcal$ that joins $p$ to itself, we conclude that $c(z)\neq0$ if and only if $W$ is not an eigenvector of $\rho(g)$.
 	\begin{figure}[!h]
 		\centering
 		\scalebox{1}{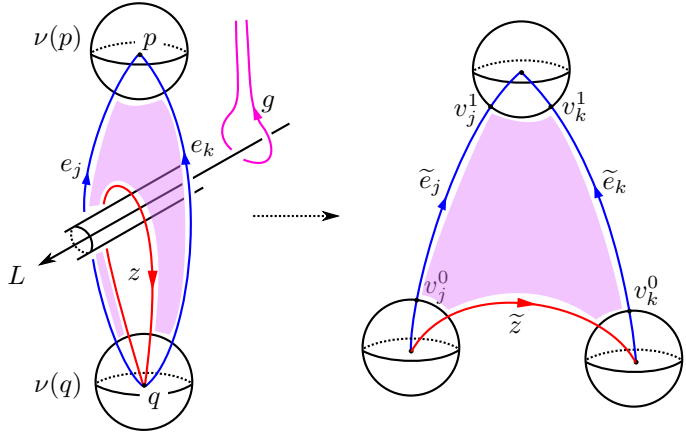} 
 		\caption{Local configuration of a lifting.}
 		\label{fig:developing2}
 	\end{figure}

 	Let us sum up the required conditions. To be precise, we enumerate the Wirtinger generators by $g_1,\cdots,g_{l}$. A desired decoration should satisfy (i) $\textrm{det}(W,V_j) \neq 0$; (ii) $W$ is not an eigenvector of $\rho(g_i)$; (iii) $V_j$ is not an eigenvector of $\rho(g_i)$
 	for all $1 \leq j \leq m$ and $1 \leq i \leq l$.
 	Since we can choose $W$ and one of $V_j$'s freely, such a decoration exists. 
 	See, for instance, Lemma 2.1 in \cite{cho2016optimistic}, Examples \ref{ex:ex1} or \ref{ex:ex2}.

\section{Volume and Chern-Simons invariant} \label{sec:VCSI}
We devote this section to prove Theorem \ref{thm:Main3}. For convenience of the reader, let us recall the theorem.

We fix a meridian $\mu_i$ and let $\lambda_i$ be the canonical longitude of each component of a link $L$. For  $\kappa=(\kappa_1,\cdots,\kappa_h) \in (\Qbb \cup \{\infty\} )^h$ we denote by $M_\kappa$ the manifold  obtained by Dehn filling along the slope $\kappa_i$ on each boundary torus $M=S^3 \setminus \nu(L)$, where $\kappa_i=\infty$ means that we do not fill the corresponding boundary torus.

Let $\rho : \pi_1(M_\kappa) \rightarrow \psl$ be a representation.
If $M_\kappa$ has non-empty boundary, we assume that $\rho$ is boundary parabolic so that the volume and Chern-Simons invariant of $\rho$ are well-defined. Regarding $\rho$ as a representation from $\pi_1(M)$ by compositing the inclusion $\pi_1(M)\rightarrow \pi_1(M_\kappa)$, we have
\begin{equation} \label{eqn:dehn}
\left\{
\begin{array}{ll}
\textrm{tr}(\rho(\mu_i))=\pm 2,\,\, \textrm{tr}(\rho(\lambda_i))=\pm 2  & \textrm{for } \kappa_i=\infty \\[3pt]
\rho (\mu_i^{r_i}\lambda_i^{s_i})= \pm I	& \textrm{for } \kappa_i = \frac{r_i}{s_i} \neq \infty
\end{array}
\right.
\end{equation} where $r_i$ and $s_i$ are coprime integers.
If we assume that $\rho: \pi_1(M)\rightarrow \psl$ admits a $\sl$-lifting and $\rho(\mu_i) \neq \pm I$ for all $1 \leq  i \leq h$, then there exists a point $(\mathbf{w},\mathbf{m})$ such that $\rho_{\mathbf{w},\mathbf{m}} = \rho$ up to conjugation where  $m_i$ is an eigenvalue of $\rho(\mu_i)$. Recall Corollary \ref{cor:main2} and Theorem \ref{thm:main1}. It follows from the equation (\ref{eqn:dehn}) that for $\kappa_i \neq \infty$ we have $m_i^{r_i} l_i^{s_i}=\pm1$ and thus $r_i\textrm{log}\, m_i+s_i\textrm{log}\, l_i \equiv 0$ in modulo $\pi \sqrt{-1}$ where $l_i$ is an eigenvalue $\rho(\lambda_i)$. From coprimeness of $(r_i,s_i)$, there exists integers  $u_i$ and $v_i$ satisfying
\begin{equation} \label{eqn:rs}
r_i \textrm{log}\,m_i + s_i \textrm{log}\,l_i + \pi \sqrt{-1}( r_i u_i + s_i v_i)=0.
\end{equation}

\begin{thm}[Theorem \ref{thm:Main3}] \label{thm:main3}  The volume and Chern-Simons invariant of $\rho : \pi_1(M_\kappa)\rightarrow \psl$ are given by \[\sqrt{-1}(\textrm{Vol}(\rho)+ \sqrt{-1}\mkern1mu\textrm{CS}(\rho)) \equiv \Wbb_0(\mathbf{w},\mathbf{m}) \quad \textrm{mod } \pi^2 \Zbb\] where the function $\Wbb_0(w_1,\cdots,w_{n},m_1,\cdots,m_h)$ is defined by
	\begin{equation*}
	\begin{array}{l}
	\Wbb_0:=\Wbb(w_1,\cdots,w_{n},m_1,\cdots,m_h)-\displaystyle\sum_{j=1}^{n}  \mkern-2mu \left(w_j \dfrac{\partial \Wbb}{\partial w_j}\right)\textrm{log}\,w_j\\[10pt]
	\quad \quad \quad  - \mkern -3mu \displaystyle\sum_{ \kappa_i\neq \infty} \mkern-3mu \left[\left(m_i \dfrac{\partial \Wbb}{\partial m_i}\right)(\textrm{log}\,m_i + u_i \pi \sqrt{-1}) -\dfrac{r_i}{s_i}(\textrm{log}\,m_i + u_i \pi \sqrt{-1})^2 \right].
	\end{array}
	\end{equation*}
\end{thm}

\subsection{Proof of Theorem \ref{thm:main3}}

We assign a vertex-ordering of each tetrahedron $\Delta$ of $\Tcal$ as in Figure \ref{fig:octahedron}. Note that these orderings agree on the common faces, so we may orient every edge of $\Tcal$ with respect to this vertex-ordering. 
We say that $\Delta$ is \emph{positively oriented} if the orientation of $\Delta$ induced from the vertex-ordering agrees with the orientation of $N$, and $\Delta$ is \emph{negatively oriented}, otherwise. We let $\epsilon_\Delta = \pm1$ according to this orientation of $\Delta$.

Let $\widetilde{\rho} : \pi_1(N)\rightarrow \sl$ be a lifting of $\rho$ and $\sigma : \partial N^1 \rightarrow \Cbb^\times$ be a cocycle assoicated to $\widetilde{\rho}$ which is trivial on the sphere components. From the equation (\ref{eqn:dehn}) we have
\begin{equation*} 
\left\{
\begin{array}{ll}
\overline{\sigma}(\mu_i)=\pm 1,\,\, \overline{\sigma}(\lambda_i)=\pm 1  & \textrm{for } \kappa_i=\infty \\[3pt]
\overline{\sigma} (\mu_i^{r_i}\lambda_i^{s_i})= \pm 1	& \textrm{for } \kappa_i = \frac{r_i}{s_i} \neq \infty
\end{array}
\right. .
\end{equation*}
It is showed in \cite{yoon2018volume} (seee Proposition 3.12 in \cite{yoon2018volume}) that there exists a cocycle $a : \partial N^1 \rightarrow \Cbb$ such that (i) $a(e) \equiv \textrm{log}\, \sigma(e)$ in modulo $\pi \sqrt{-1} \Zbb$ for all $e \in \partial N^1$; (ii) $a$ is trivial on the sphere components; (iii) the induced homomorphism $\overline{a}$ satisfies 
\begin{equation*} 
\left\{
\begin{array}{ll}
\overline{a}(\mu_i)=\overline{a}(\lambda_i)=0  & \textrm{for } \kappa_i=\infty \\[3pt]
\overline{a}(\mu_i) = \textrm{log}\,\overline{\sigma}(\mu_i) + u_i \pi \sqrt{-1} \textrm{ and }  \overline{a}(\lambda_i) = \textrm{log}\, \overline{\sigma}(\lambda_i) + v_i \pi \sqrt{-1} & \textrm{for } \kappa_i \neq \infty
\end{array}
\right. .
\end{equation*} The equation  (\ref{eqn:rs}) tells us that 
$r_i\overline{a}(\mu_i)+s_i\overline{a}(\lambda_i)=0$ for all $\kappa_i\neq \infty$.
On the other hand, by Theorem \ref{thm:main2} there exists a $\sigma$-deformed Ptolemy assignment $c : \Tcal^1 \rightarrow \Cbb^\times$ such that $\rho_c = \widetilde{\rho}$ up to conjugation. We let \[(\mathbf{w},\mathbf{m})=(c(e_1),\cdots, c(e_n), \overline{\sigma}(\mu_1),\cdots,\overline{\sigma}(\mu_h))\] as in Proposition \ref{prop:key}.

For each ideal tetrahedron $\Delta$ (with edges denoted as in Figure \ref{fig:tetrahedron}) of $\Tcal$, we let 
\begin{align*}
	z&=\dfrac{\sigma(s_{12})\sigma(s_{45})}{\sigma(s_{24})\sigma(s_{51})}\,\dfrac{c(l_1) c(l_4)}{c(l_2)c(l_5)} \\
	p\pi \sqrt{-1}&=(a(s_{12}) + a(s_{45}) - a(s_{24})-a(s_{51}) \\ & \quad + \textrm{log}\,c(l_1)+ \textrm{log}\,c(l_4)-\textrm{log}\,c(l_2)-\textrm{log}\,c(l_5) - \textrm{log}\, z \\
	q\pi \sqrt{-1} &= a(s_{53})+a(s_{26})-a(s_{32}) - a(s_{65}) \\ & \quad + \textrm{log}\,c(l_2) +\textrm{log}\,c(l_5)-\textrm{log}\,c(l_3) -\textrm{log}\,c(l_6) + \textrm{log}\,(1-z)
\end{align*}
if $\epsilon_\Delta =1$ and
\begin{align*}
	z&=\dfrac{\sigma(s_{24})\sigma(s_{51})}{\sigma(s_{12})\sigma(s_{45})}\,\dfrac{c(l_2)c(l_5)}{c(l_1) c(l_4)} \\
	p \pi \sqrt{-1}&=  a(s_{24})+a(s_{51}) -a(s_{12}) - a(s_{45}) \\ & \quad + \textrm{log}\,c(l_2)+ \textrm{log}\,c(l_5)-\textrm{log}\,c(l_1)-\textrm{log}\,c(l_4) - \textrm{log}\, z\\
	q \pi \sqrt{-1} &= a(s_{43})+a(s_{16})-a(s_{64}) - a(s_{31}) \\ & \quad + \textrm{log}\,c(l_1) +\textrm{log}\,c(l_4)-\textrm{log}\,c(l_3) -\textrm{log}\,c(l_6) + \textrm{log}\,(1-z)
\end{align*}
if $\epsilon_\Delta =-1$, and let $R(\Delta) := R(z;p,q)$ where $R$ is the extended Rogers dilogarithm \cite{neumann2004extended} defined by
\begin{equation*} 
R(z;p,q)=\textrm{Li}_2(z)+ \frac{\pi \sqrt{-1}}{2}( p\, \textrm{log}\,(1-z)+ q\, \textrm{log}\,z )+\frac{1}{2} \textrm{log}\,(1-z)\,\textrm{log} \, z - \frac{\pi^2}{2}.
\end{equation*} 
It is proved in \cite{yoon2018volume} (see Theorem 3.4 and Remark 3.5 in \cite{yoon2018volume}) that 
\begin{equation} \label{eqn:ref} \sqrt{-1}(\textrm{Vol}(\rho)+ \sqrt{-1}\mkern1mu\textrm{CS}(\rho)) \equiv \sum_\Delta \epsilon_\Delta R(\Delta) \quad \textrm{mod } \pi^2 \Zbb
\end{equation}
 where the sum is over all tetrahedra $\Delta$ of $\Tcal$. We refer \cite{yoon2018volume} for details. Therefore, it is enough to show that the right-hand side of the equation (\ref{eqn:ref}) is equal to $\Wbb_0 (\mathbf{w},\mathbf{m})$ in modulo $\pi^2 \Zbb$.
 
Let us first consdier a crossing of $L$ as in Figure~\ref{fig:crossing}(a). At this crossing, we denote edges of $\Tcal$ as in Figure~\ref{fig:octahedron}(a). We also denote by $\Delta^1$ the tetrahedron corresponding to the edge $h^1$ as in Figure \ref{fig:cross_ratio_pos}, and denote similarly for $h^3,h^5,h_1,$ and $h_3$. It is not hard to check that 
$\epsilon_{\Delta^1}=\epsilon_{\Delta^5}=\epsilon_{\Delta_1}=1$ and $\epsilon_{\Delta^3}=\epsilon_{\Delta_3}=-1$.
A straightforward computation gives 
\begin{talign*}
&R(\Delta^1)  = \textrm{Li}_2(\frac{w_m}{m_\beta w_j}) - \frac{\pi^2}{6}+\frac{1}{2}(\textrm{log}\,w_m - \textrm{log}\,w_j - \overline{a}(\mu_\beta) )\, \textrm{log}(1-\frac{w_m}{m_\beta w_j}) \\
&\quad +\frac{1}{2}(\textrm{log}\,w_j  - \textrm{log}\,c(h^5) + \textrm{log}\,c(h^2) - \textrm{log}\,c(h^1)  + 
 a(s^{41})+\textrm{log}(1-\frac{w_m}{m_\beta w_j})) \textrm{log}\,\frac{w_m}{m_\beta w_j}.
\end{talign*} Since $\textrm{log}\, \frac{w_m}{m_\beta w_j} \equiv \textrm{log}\,w_m - \textrm{log}\,w_j - \overline{a}(\mu_\beta)$ in modulo $2 \pi \sqrt{-1}$,
\begin{talign*}
&R(\Delta^1) \equiv \textrm{Li}_2(\frac{w_m}{m_\beta w_j}) - \frac{\pi^2}{6}+(\textrm{log}\,w_m - \textrm{log}\,w_j - \overline{a}(\mu_\beta))\, \textrm{log}(1-\frac{w_m}{m_\beta w_j}) \\
&\quad+\frac{1}{2}(\textrm{log}\,w_j   - \textrm{log}\,c(h^5)+ \textrm{log}\,c(h^2) - \textrm{log}\,c(h^1)  +  a(s^{41}))(\textrm{log}\,w_m - \textrm{log}\,w_j - \overline{a}(\mu_\beta))
\end{talign*} in modulo $\pi^2 \Zbb$.  We similarly compute the Rogers dilogarithm terms for other tetrahedra and obtain : 
\allowdisplaybreaks
\begin{talign*}
&R(\Delta^1)-R(\Delta^3)+R(\Delta_1)-R(\Delta_3)+R(\Delta^5) \\
&=\textrm{Li}_2(\frac{w_m}{m_\beta w_j})- \textrm{Li}_2(\frac{w_l}{m_\beta w_k})+\textrm{Li}_2(\frac{w_k}{m_\alpha w_j})-\textrm{Li}_2(\frac{w_l}{m_\alpha w_m})+\textrm{Li}_2(\frac{w_j w_l}{w_k w_m})  -\frac{\pi^2}{6}\\
&+ (\textrm{log}\,w_m - \textrm{log}\,w_j - \overline{a}(\mu_\beta) )\, \textrm{log}(1-\frac{w_m}{m_\beta w_j})\\
&+ (\textrm{log}\,w_k - \textrm{log}\,w_l + \overline{a}(\mu_\beta))\, \textrm{log}(1-\frac{w_l}{m_\beta w_k})\\
&+ (\textrm{log}\,w_k - \textrm{log}\,w_j - \overline{a}(\mu_\alpha) )\,\textrm{log}(1-\frac{w_k}{m_\alpha w_j})\\
&+ (\textrm{log}\,w_m - \textrm{log}\,w_l + \overline{a}(\mu_\alpha) )\, \textrm{log}(1-\frac{w_l}{m_\alpha w_m})\\
&+ (\textrm{log}\,w_l + \textrm{log}\,w_j - \textrm{log}\,w_k-\textrm{log}\,w_m)\, \textrm{log}(1-\frac{w_j w_l}{w_k w_m})\\
&+\frac{1}{2} (\textrm{log}\,w_j-\textrm{log}\,c(h^5)+\textrm{log}\,c(h^2)-\textrm{log}\,c(h^1)+ a(s^{41}))(\textrm{log}\,w_m-\textrm{log}\,w_j -\overline{a}(\mu_\beta))\\
&+\frac{1}{2} (\textrm{log}\,w_k-\textrm{log}\,c(h^5)+\textrm{log}\,c(h^2)-\textrm{log}\,c(h^3)+a(s^{43}))(\textrm{log}\,w_k-\textrm{log}\,w_l+\overline{a}(\mu_\beta))\\
&+\frac{1}{2} (  \textrm{log}\,w_j -\textrm{log}\,c(h_5) + \textrm{log}\,c(h_2)-\textrm{log}\,c(h_1) + a(s_{41}))(\textrm{log}\,w_k-\textrm{log}\,w_j-\overline{a}(\mu_\alpha)) \\
&+\frac{1}{2} ( \textrm{log}\,w_m-\textrm{log}\,c(h_5) + \textrm{log}\,c(h_2) -\textrm{log}\,c(h_3) + a(s_{43}))(\textrm{log}\,w_m-\textrm{log}\,w_l+\overline{a}(\mu_\alpha))\\
&+\frac{1}{2} (\textrm{log}\, w_k + \textrm{log}\,w_m - \textrm{log}\, c(h^5) -\textrm{log}\,c(h_5) ) (\textrm{log}\,w_l+\textrm{log}\,w_j-\textrm{log}\,w_k-\textrm{log}\,w_m)
\end{talign*} 
Rearranging the last five lines appropriately, we obtain
\allowdisplaybreaks
\begin{talign*}
	&R(\Delta^1)-R(\Delta^3)+R(\Delta_1)-R(\Delta_3)+R(\Delta_5) \\
	&\left.\begin{array}{l}
		=\textrm{Li}_2(\frac{w_m}{m_\beta w_j})- \textrm{Li}_2(\frac{w_l}{m_\beta w_k})+\textrm{Li}_2(\frac{w_k}{m_\alpha w_j})-\textrm{Li}_2(\frac{w_l}{m_\alpha w_m})+\textrm{Li}_2(\frac{w_j w_l}{w_k w_m})  -\frac{\pi^2}{6}\\[5pt]
		+\, (\textrm{log}\,w_m - \textrm{log}\,w_j - \overline{a}(\mu_\beta) )\, \textrm{log}(1-\frac{w_m}{m_\beta w_j})\\[5pt]
		+\, (\textrm{log}\,w_k - \textrm{log}\,w_l + \overline{a}(\mu_\beta))\, \textrm{log}(1-\frac{w_l}{m_\beta w_k})\\[5pt]
		+\, (\textrm{log}\,w_k - \textrm{log}\,w_j - \overline{a}(\mu_\alpha) )\,\textrm{log}(1-\frac{w_k}{m_\alpha  w_j})\\[5pt]
		+\, (\textrm{log}\,w_m - \textrm{log}\,w_l + \overline{a}(\mu_\alpha) )\, \textrm{log}(1-\frac{w_l}{m_\alpha w_m})\\[5pt]
		+\, (\textrm{log}\,w_l + \textrm{log}\,w_j - \textrm{log}\,w_k-\textrm{log}\,w_m)\, \textrm{log}(1-\frac{w_j w_l}{w_k w_m})\\[5pt]
		-\,(\textrm{log}\,w_m-\textrm{log}\,w_j-\overline{a}(\mu_\beta))(\textrm{log}\,w_k-\textrm{log}\,w_j-\overline{a}(\mu_\alpha))\\
	\end{array} \right\}\textrm{A-part}\\
	&\left.\begin{array}{l}
		+\,\frac{1}{2} \textrm{log}\,c(h^2)\,(\textrm{log}\,w_m+\textrm{log}\,w_k-\textrm{log}\,w_l-\textrm{log}\,w_j) \\[5pt]
		+\,\frac{1}{2} \textrm{log}\,c(h_2)\,(\textrm{log}\,w_m+\textrm{log}\,w_k-\textrm{log}\,w_l-\textrm{log}\,w_j) \\[5pt]
		-\,\frac{1}{2} \textrm{log} \,c(h^1)\, (\textrm{log}\,w_m-\textrm{log}\,w_j -\overline{a}(\mu_\beta))\\[5pt]
		-\,\frac{1}{2} \textrm{log} \,c(h^3)\, (\textrm{log}\,w_k-\textrm{log}\,w_l +\overline{a}(\mu_\beta))\\[5pt]
		-\,\frac{1}{2} \textrm{log} \,c(h_1)\, (\textrm{log}\,w_k-\textrm{log}\,w_j -\overline{a}(\mu_\alpha))\\[5pt]
		-\,\frac{1}{2} \textrm{log} \,c(h_3)\, (\textrm{log}\,w_m-\textrm{log}\,w_l +\overline{a}(\mu_\alpha))\\
	\end{array} \right\} \textrm{B-part}\\
	&\left.\begin{array}{l}
		+\, \frac{1}{2} a(s^{41})(\textrm{log}\,w_m-\textrm{log}\,w_j)+\, \frac{1}{2} a(s^{43})(\textrm{log}\,w_k-\textrm{log}\,w_l) \\[5pt]
		+\, \frac{1}{2} a(s_{21})(\textrm{log}\,w_k-\textrm{log}\,w_j)+\, \frac{1}{2} a(s_{23})(\textrm{log}\,w_m-\textrm{log}\,w_l) \\
	\end{array} \right\} \textrm{C-part}\\
	& \left. \begin{array}{l}
		+\,\overline{a}(\mu_\alpha) \overline{a}(\mu_\beta)-\, \frac{1}{2} a(s^{31})\overline{a}(\mu_\beta)-\, \frac{1}{2} a(s_{31})\overline{a}(\mu_\alpha)\\
	\end{array} \right\} \textrm{D-part} \\
	& \left. \begin{array}{l}
		+\frac{1}{2} \overline{a}(\mu_\alpha) (\textrm{log}\,w_k-\textrm{log}\,w_l)+\frac{1}{2} \overline{a}(\mu_\beta) (\textrm{log}\,w_j-\textrm{log}\,w_k)\\
	\end{array} \right\} \textrm{E-part} 
\end{talign*} Note that in the above computation, $\textrm{log}\,c(h^5)$- and $\textrm{log}\,c(h_5)$-terms vanish, and we replace $a(s_{41})$ and $a(s_{43})$ by $\overline{a}(\mu_\alpha)+a(s_{21})$ and $\overline{a}(\mu_\alpha)+a(s_{23})$, respectively.
We compute similarly for a crossing as in Figure \ref{fig:crossing}(b) and obtain:
\begin{talign*}
	&-R(\Delta^1)+R(\Delta^3)-R(\Delta_1)+R(\Delta_3)-R(\Delta^5) \\
	&\left.\begin{array}{l}
		=-\textrm{Li}_2(\frac{m_\beta w_m}{w_j})+ \textrm{Li}_2(\frac{m_\beta w_l}{w_k})-\textrm{Li}_2(\frac{m_\alpha w_k}{ w_j})+\textrm{Li}_2(\frac{m_\alpha w_l}{w_m})-\textrm{Li}_2(\frac{w_j w_l}{w_k w_m})  +\frac{\pi^2}{6}\\[5pt]
		-\, (\textrm{log}\,w_m - \textrm{log}\,w_j + \overline{a}(\mu_\beta) )\, \textrm{log}(1-\frac{m_\beta w_m}{ w_j})\\[5pt]
		-\, (\textrm{log}\,w_k - \textrm{log}\,w_l - \overline{a} (\mu_\beta))\, \textrm{log}(1-\frac{m_\beta w_l}{w_k})\\[5pt]
		-\, (\textrm{log}\,w_k - \textrm{log}\,w_j + \overline{a}(\mu_\alpha) )\,\textrm{log}(1-\frac{m_\alpha w_k}{ w_j})\\[5pt]
		-\, (\textrm{log}\,w_m - \textrm{log}\,w_l - \overline{a}(\mu_\alpha) )\, \textrm{log}(1-\frac{m_\alpha w_l}{ w_m})\\[5pt]
		-\, (\textrm{log}\,w_l + \textrm{log}\,w_j - \textrm{log}\,w_k-\textrm{log}\,w_m)\, \textrm{log}(1-\frac{w_j w_l}{w_k w_m})\\[5pt]
		+\,(\textrm{log}\,w_m-\textrm{log}\,w_j+\overline{a}(\mu_\beta))(\textrm{log}\,w_k-\textrm{log}\,w_j+\overline{a}(\mu_\alpha))\\
	\end{array} \right\}\textrm{A-part}\\
	&\left.\begin{array}{l}
		-\,\frac{1}{2}\textrm{log}\,c(h^2)\,(\textrm{log}\,w_m+\textrm{log}\,w_k-\textrm{log}\,w_l-\textrm{log}\,w_j) \\[5pt]
		-\,\frac{1}{2}\textrm{log}\,c(h_2)\,(\textrm{log}\,w_m+\textrm{log}\,w_k-\textrm{log}\,w_l-\textrm{log}\,w_j) \\[5pt]
		+\,\frac{1}{2} \textrm{log} \,c(h_1)\, (\textrm{log}\,w_m-\textrm{log}\,w_j +\overline{a}(\mu_\beta))\\[5pt]
		+\,\frac{1}{2} \textrm{log} \,c(h_3)\, (\textrm{log}\,w_k-\textrm{log}\,w_l -\overline{a}(\mu_\beta))\\[5pt]
		+\,\frac{1}{2} \textrm{log} \,c(h^1)\, (\textrm{log}\,w_k-\textrm{log}\,w_j +\overline{a}(\mu_\alpha))\\[5pt]
		+\,\frac{1}{2} \textrm{log} \,c(h^3)\, (\textrm{log}\,w_m-\textrm{log}\,w_l -\overline{a}(\mu_\alpha))\\
	\end{array} \right\} \textrm{B-part}\\
	&\left.\begin{array}{l}
		-\, \frac{1}{2} a(s^{41})(\textrm{log}\,w_k-\textrm{log}\,w_j)-\, \frac{1}{2} a(s^{43})(\textrm{log}\,w_m-\textrm{log}\,w_l) \\[5pt]
		-\, \frac{1}{2} a(s_{21})(\textrm{log}\,w_m-\textrm{log}\,w_j)-\, \frac{1}{2} a(s_{23})(\textrm{log}\,w_l-\textrm{log}\,w_k) \\
	\end{array} \right\} \textrm{C-part}\\
	& \left. \begin{array}{l}
		-\,\overline{a}(\mu_\alpha)\overline{a}(\mu_\beta)-\, \frac{1}{2} a(s^{31})\overline{a}(\mu_\alpha)-\, \frac{1}{2} a(s_{31})\overline{a}(\mu_\beta)\\
	\end{array} \right\} \textrm{D-part} \\
	& \left. \begin{array}{l}
		+\frac{1}{2} \overline{a}(\mu_\beta) (\textrm{log}\,w_j-\textrm{log}\,w_k)+\frac{1}{2} \overline{a}(\mu_\alpha) (\textrm{log}\,w_k-\textrm{log}\,w_l)\\
	\end{array} \right\} \textrm{E-part} 
\end{talign*}
As one can see, we divide the Rogers dilogarithm terms coming from a crossing into 5 parts: A, B, C, D, and E-parts.

Let us first consider A-parts. If we use the equality
\begin{talign*}
	-&(\textrm{log}\,w_k-\textrm{log}\,w_j-\overline{a}(\mu_\alpha))(\textrm{log}\,w_m-\textrm{log}\,w_j-\overline{a}(\mu_\beta)) \\
	=&-(\textrm{log}\,w_k-\textrm{log}\,w_j-\overline{a}(\mu_\alpha)-\textrm{log}\frac{w_k}{m_\alpha w_j})(\textrm{log}\,w_m-\textrm{log}\,w_j-\overline{a}(\mu_\beta)) \\
	& \quad\quad\quad -\textrm{log}\frac{w_k}{m_\alpha w_j}(\textrm{log}\,w_m-\textrm{log}\,w_j-\overline{a} (\mu_\beta)) \\
	\equiv& -\,(\textrm{log}\,w_k-\textrm{log}\,w_j-\overline{a}(\mu_\alpha)-\textrm{log}\,\frac{w_k}{m_\alpha w_j})\textrm{log}\, \frac{w_m}{m_\beta w_j} \\
	&  \quad \quad \quad -\textrm{log}\,\frac{w_k}{m_\alpha w_j}(\textrm{log}\,w_m-\textrm{log}\,w_j-\overline{a}(\mu_\beta))  \quad\quad \quad \quad (\textrm{mod } \pi^2 \Zbb),
\end{talign*}  then one can directly check that the sum of A-parts over all crossings is equal to 
$$\Wbb(\mathbf{w},\mathbf{m})-\displaystyle\sum_{j=1}^{n}  \mkern-5mu \left(w_j \dfrac{\partial \Wbb}{\partial w_j}\right)\textrm{log}\,w_j- \displaystyle\sum_{i=1}^{h} \mkern-5mu \left(m_i \dfrac{\partial \Wbb}{\partial m_i}\right) \overline{a}(\mu_i).$$ 

For D-parts, the sum of $-\frac{1}{2} a(s^{31})\overline{a}(\mu_i)$-terms along the $i$-th component of $L$ results in $-\frac{1}{2} \overline{a}(\lambda_{i;bf})\overline{a}(\mu_i)$, where $\lambda_{i;bf}$ is the blackboard framed longitude of the $i$-th component. Similarly, the sum of $-\frac{1}{2} a(s_{31})b_i(\mu_i)$-terms also results in $-\frac{1}{2} \overline{a}(\lambda_{i;bf})\overline{a}(\mu_i)$. The remaining terms $\pm \overline{a}(\mu_i)\overline{a}(\mu_j)$ revise the framing appropriately and so the sum of D-parts over all crossings is equal to $$-\displaystyle\sum_{i=1}^{h} \mkern-3mu \overline{a}(\mu_i)\overline{a}(\lambda_i).$$ 

\begin{lemma} \label{lem:bpart} The sum of B-parts over all crossings vanishes.
\end{lemma}
\begin{proof} Let $e$ be an over edge of $\Tcal$ with the corresponding over-arc of $L$ as in Figure \ref{fig:local_diagram}(a). Note that the edge $e$ appears as $h_1$ at the initial crossing, as $h_3$ at the terminal crossing, and as $h^2=h^4$ at the intermediate crossings. 
Then, in the sum of B-parts, $\textrm{log}\, c(e)$-terms appear exactly at these crossings and their sum is given by 
\allowdisplaybreaks
\begin{talign*}
		&\frac{1}{2} \textrm{log}\, c(e) \Big[(-\textrm{log}\, w_{j_1}+\textrm{log}\, w_{j_2} -\overline{a}(\mu_i))\\
		&\quad+(\textrm{log}\, w_{j_1}-\textrm{log}\, w_{j_2}-\textrm{log}\, w_{j_3}+\textrm{log}\, w_{j_4})+\cdots  \\[2pt]
		 &\quad 	+(\textrm{log}\, w_{j_{2m-1}}-\textrm{log}\, w_{j_{2m}}-\textrm{log}\, w_{j_{2m+1}}+\textrm{log}\, w_{j_{2m+2}})\\ 
		& \quad +(\textrm{log}\, w_{j_{2m+1}}-\textrm{log}\, w_{j_{2m+2}}+\overline{a}(\mu_i))\Big]	=0.
\end{talign*} Note that changing orientations that are not specified in the local diagram dose not change the computation. We compute similarly for an under edge of $\Tcal$, and complete the proof.
\end{proof}
We omit a proof the fact that the sum of $D$-parts and $E$-parts are respectively zero, since it can be checked combinatorially as in Lemma \ref{lem:bpart}. 

Recall that we have $\overline{a}(\mu_i) = 0$ for $\kappa_i=\infty$  and  $r_i\overline{a}(\mu_i)+s_i\overline{a}(\lambda_i)=0$ for $\kappa_i \neq \infty$. It thus follows that the sum of A- and D-parts over all crossings is equal to $\mathbb{W}_0(\mathbf{w},\mathbf{m})$. This completes the proof, since the sums of B-, C-, and E-parts are all zero.

\begin{exam} \label{ex:ex1} We consider a diagram of the figure-eight knot and denote the Wirtinger generators by $g_1,\cdots,g_4$ as in Figure \ref{fig:figure_eight_diagram}. 
	\begin{figure}[!h]
		\centering
		\scalebox{1}{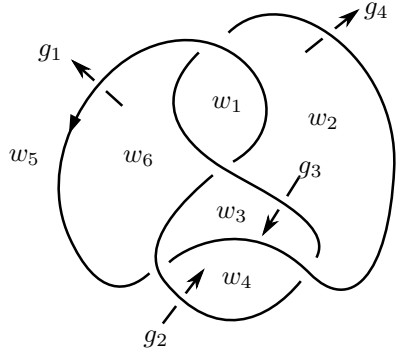} 
		\caption{The figure eight knot diagram.}
		\label{fig:figure_eight_diagram}
	\end{figure} 
It is known that  $$\rho(g_1)=\begin{pmatrix} m&1\\0 &m^{-1}\end{pmatrix} \textrm{ and }\rho(g_4)=\begin{pmatrix} m&0\\y &m^{-1}\end{pmatrix}$$ determine a $\sl$-representation $\rho$ of the knot group if $$y=\frac{-m^4+3 m^2-1 +\sqrt{m^8-2 m^6-m^4-2 m^2+1}}{2 m^2}.$$  The canonical longitude $\lambda$ of the knot is given by $g_2 \, g_4^{-1} \, g_2^{-1} \, g_1$, so an eignvalue $l$ of $\rho(\lambda)$ is given by
$$l=\frac{m^8-m^6+2m^4-m^2+1+(m^4-1)\sqrt{m^8-2 m^6-m^4-2 m^2+1}}{2 m^4}.$$
If we consider the $\frac{2}{3}$-Dehn filling, then we require $m \in \Cbb^\times$ satisfying $m^2l^3=1$; using the Mathematica, we have $$(m,l)=(-1.30664 +0.04987\sqrt{-1},\ -0.43642 + 0.71337\sqrt{-1}).$$ 
We remark that the representation $\rho$ is in fact (a lifting of) the geometric representation for the  $\frac{2}{3}$-filled manifold $M_{\frac{2}{3}}$ obtained from the figure-eight knot exterior.
We let $(u,v)=(-2,0)$ so that $$2 \, \textrm{log}\,m +3\,\textrm{log}\,l +\pi \sqrt{-1}(2 u + 3v )=0.$$

We now consider the vectors $V_j$'s, each of which corresponds to a region, as in Section \ref{sec:proof2}. Recall that these vectors satisfy the condition $$V_{j} =\rho(g_k)^{-1} V_{i}$$ at each arc as in Figure \ref{fig:rule}. (cf. region coloring in \cite{carter2001geometric,cho2016optimistic}.) Note that they are well-determined whenever an initial vector is chosen arbitrarily. For instance, if we choose $V_6= \binom{1}{\sqrt{-1}}$, then we have
\begin{equation*}
\begin{array}{ll}
V_1=\dbinom{-0.84795 - 1.60327 \sqrt{-1}}{-0.44863 - 0.05668 \sqrt{-1}},& V_2=\dbinom{1.04988 + 
	1.30664 \sqrt{-1}}{0.58903 + 0.05168 \sqrt{-1}},\\[10pt]
V_3=\dbinom{-0.78470 + 
	0.37242 \sqrt{-1}}{-0.39208 - 1.12719 \sqrt{-1}},& V_4=\dbinom{0.61054 - 
	0.26172 \sqrt{-1}}{1.12129 + 1.96967 \sqrt{-1}},\\[10pt]
V_5=\dbinom{-0.76421 - 
	1.02917 \sqrt{-1}}{-0.04987 - 1.30664 \sqrt{-1}},& V_6=\dbinom{1}{\sqrt{-1}}.
\end{array}
\end{equation*}
We also choose another vector $W$ almost arbitrarily; for instance, we let $W=\binom{2}{1}$. Then we have $\mathbf{w}=(w_1,\cdots,w_6)$ by $w_j=\textrm{det}(W,V_j)$:
\begin{equation*}
\begin{array}{ll}
w_1=-0.04931 + 1.48991 \sqrt{-1},& w_2=0.12818 - 1.20327 \sqrt{-1},\\[3pt]
w_3 =0.00054 - 2.62681 \sqrt{-1},& w_4=1.63204 + 4.20107 \sqrt{-1},\\[3pt]
w_5= 0.66446 - 1.58411 \sqrt{-1},& w_6= -1 + 2 \sqrt{-1}.
\end{array}
\end{equation*}

\begin{figure}[!h]
	\centering
	\scalebox{1}{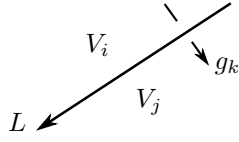}
	\caption{Rule for a region coloring.}
	\label{fig:rule}
\end{figure}

  Plugging the above non-degenerate solution $(\mathbf{w},\mathbf{m})=(w_1,\cdots,w_6,m)$ to Theorem \ref{thm:main2}, we obtain
	$$\sqrt{-1} \big( \textrm{Vol}(M_{\frac{2}{3}}) + \sqrt{-1}\, \textrm{CS}(M_\frac{2}{3})\big)=-3.33836 + 1.73712 \sqrt{-1}.$$ 
	Note that changing choices for $V_6$ and $V_0$ may give a different non-degenerate solution but it results in the same volume and Chern-Simons invariant.

	\begin{exam} \label{ex:ex2}Let us consider a diagram of the Whitehead link as in Figure \ref{fig:white}. One can check that  $$\rho(g_1)=\begin{pmatrix} m_1&1\\0 &m_1^{-1}\end{pmatrix} \textrm{ and }\rho(g_2)=\begin{pmatrix} m_2&0\\y &m_2^{-1}\end{pmatrix}$$ determine a $\sl$-representation of the link group if 
		\begin{talign*}
			&m_1 m_2 (m_1^2-1)(m_2^2-1)+((m_1^2 m_2^2+1)(m_1^2-1)(m_2^2-1)+2m_1^2 m_2^2)y\\
			&+(2-m_1^2-m_2^2+2m_1^2m^2_2)m_1 m_2 y^2+m_1^2 m_2^2y^3=0.
		\end{talign*}
		\begin{figure}[!h]
		\centering
		\scalebox{1}{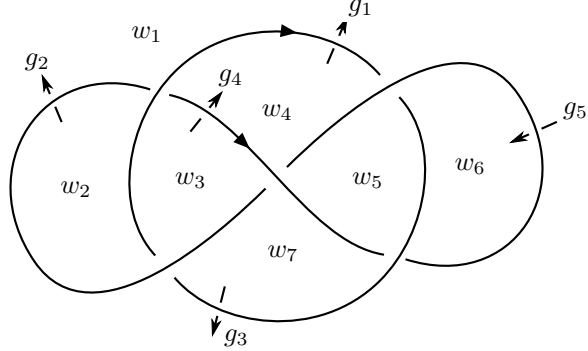} 
		\caption{The Whitehead link.}
		\label{fig:white}
	\end{figure}
 	The longitude of the circular component is given by $g_5 g_2^{-1}$ and that of the other component is given by $g_2 g_1^{-1} g_3 g_4^{-1}$. We obtain
			\allowdisplaybreaks
			\begin{talign*}
				l_1 &= \frac{1}{m_1^2 m_2^3} \Big[m_1^4 m_2 (m_2^2-1)^2 y^2+m_1^3 (m_2^2-1)  (2 m_2^2 (y^2+1)-1)y\\
				&\quad +m_1^2(-m_2^5y^2+m_2^3 (y^4+5 y^2+1)-3 m_2 y^2)\\
				&\quad -m_1 y (m_2^4 (m_2^2+1)-2 m_2^2 (y^2+1)+1)-m_2 (m_2^2-1) y^2\Big],\\
				l_2 &=\frac{1}{m_1^3 m_2}\Big[m_1^4 m_2^2 y+m_1^3 \left(-m_2^3 y^2+m_2 y^2+m_2\right)\\
				&+m_1^2 \left(-m_2^2 y^3-2 m_2^2 y+y\right)+m_1 m_2 \left(m_2^2-2\right) y^2+\left(m_2^2-1\right) y\Big].
			\end{talign*}
		Let us consider $\kappa=(-5,-\frac{5}{2})$ filling; using Mathematica, one can check that  
		\begin{talign*}
			(m_1,l_1)&=(0.60430 + 1.35917 \sqrt{-1},\ 6.31525 - 3.62462 \sqrt{-1})\\
			(m_2,l_2)&=(1.43249 + 1.08047 \sqrt{-1},\ -4.30814 -0.19296 \sqrt{-1})
		\end{talign*} satisfies (numerically) $m_i^{r_i}l_i^{s_i}=1$ for $i=1,2$. We let $(u_1,v_1)=(0,2)$ and  $(u_2,v_2)=(-1,-1)$ so that  the equation (\ref{eqn:rs}) holds for $i=1,2$.
	
		Choosing an initial vector $V_1=\binom{1}{\sqrt{-1}}$ and $W=\binom{2}{1}$, we  obtain :
		\begin{equation*}
			\begin{array}{ll}
				w_1=-1 + 2 \sqrt{-1}, & w_2=	1.93847 - 5.78499 \sqrt{-1}, \\[2pt]
				w_3= -3.05190 - 3.60342 \sqrt{-1}, & w_4=	0.62430 - 1.81291 \sqrt{-1},\\[2pt]
				w_5= -0.59085 - 0.74757 \sqrt{-1}, & w_6=-1.23298 + 2.38517 \sqrt{-1},\\[2pt]
				w_7= -4.06837 - 1.29382 \sqrt{-1}&
			\end{array}
		\end{equation*} Plugging the above non-degenerate solution $(\mathbf{w},\mathbf{m})=(w_1,\cdots,w_7,m_1,m_2)$ to Theorem \ref{thm:main2}, we obtain
		$$\sqrt{-1} \big( \textrm{Vol}(M_\kappa) + \sqrt{-1}\, \textrm{CS}(M_\kappa)\big)=
	     1.18520 + 0.94270\sqrt{-1}.$$
	\end{exam}

%
\end{exam}

\bibliographystyle{abbrv}
\bibliography{biblog}
\end{document}

%% file: crossing.pdf_tex
\begingroup%
  \makeatletter%
  \providecommand\color[2][]{%
    \errmessage{(Inkscape) Color is used for the text in Inkscape, but the package 'color.sty' is not loaded}%
    \renewcommand\color[2][]{}%
  }%
  \providecommand\transparent[1]{%
    \errmessage{(Inkscape) Transparency is used (non-zero) for the text in Inkscape, but the package 'transparent.sty' is not loaded}%
    \renewcommand\transparent[1]{}%
  }%
  \providecommand\rotatebox[2]{#2}%
  \ifx\svgwidth\undefined%
    \setlength{\unitlength}{257.37954207bp}%
    \ifx\svgscale\undefined%
      \relax%
    \else%
      \setlength{\unitlength}{\unitlength * \real{\svgscale}}%
    \fi%
  \else%
    \setlength{\unitlength}{\svgwidth}%
  \fi%
  \global\let\svgwidth\undefined%
  \global\let\svgscale\undefined%
  \makeatother%
  \begin{picture}(1,0.30837757)%
    \put(0,0){\includegraphics[width=\unitlength,page=1]{crossing.pdf}}%
    \put(0.1496553,0.1731147){\color[rgb]{0,0,0}\makebox(0,0)[lb]{\smash{$w_m$}}}%
    \put(0.2376992,0.09686963){\color[rgb]{0,0,0}\makebox(0,0)[lb]{\smash{$w_j$}}}%
    \put(0.31660461,0.17164215){\color[rgb]{0,0,0}\makebox(0,0)[lb]{\smash{$w_k$}}}%
    \put(0.23629628,0.24249502){\color[rgb]{0,0,0}\makebox(0,0)[lb]{\smash{$w_l$}}}%
    \put(0.08476714,0.00779544){\color[rgb]{0,0,0}\makebox(0,0)[lb]{\smash{(a) Positive crossing}}}%
    \put(0,0){\includegraphics[width=\unitlength,page=2]{crossing.pdf}}%
    \put(0.64569078,0.17456354){\color[rgb]{0,0,0}\makebox(0,0)[lb]{\smash{$w_m$}}}%
    \put(0.73180506,0.09831828){\color[rgb]{0,0,0}\makebox(0,0)[lb]{\smash{$w_j$}}}%
    \put(0.81264032,0.17405605){\color[rgb]{0,0,0}\makebox(0,0)[lb]{\smash{$w_k$}}}%
    \put(0.73040224,0.24394367){\color[rgb]{0,0,0}\makebox(0,0)[lb]{\smash{$w_l$}}}%
    \put(0.56523315,0.00924409){\color[rgb]{0,0,0}\makebox(0,0)[lb]{\smash{(b) Negative crossing}}}%
    \put(0,0){\includegraphics[width=\unitlength,page=3]{crossing.pdf}}%
    \put(0.1173476,0.29030124){\color[rgb]{0,0,0}\makebox(0,0)[lb]{\smash{$m_{\alpha}$}}}%
    \put(0.33265902,0.28999638){\color[rgb]{0,0,0}\makebox(0,0)[lb]{\smash{$m_{\beta}$}}}%
    \put(0.61925321,0.28969701){\color[rgb]{0,0,0}\makebox(0,0)[lb]{\smash{$m_{\alpha}$}}}%
    \put(0.83017411,0.28861603){\color[rgb]{0,0,0}\makebox(0,0)[lb]{\smash{$m_{\beta}$}}}%
  \end{picture}%
\endgroup%

%% file: cross.pdf_tex
\begingroup%
  \makeatletter%
  \providecommand\color[2][]{%
    \errmessage{(Inkscape) Color is used for the text in Inkscape, but the package 'color.sty' is not loaded}%
    \renewcommand\color[2][]{}%
  }%
  \providecommand\transparent[1]{%
    \errmessage{(Inkscape) Transparency is used (non-zero) for the text in Inkscape, but the package 'transparent.sty' is not loaded}%
    \renewcommand\transparent[1]{}%
  }%
  \providecommand\rotatebox[2]{#2}%
  \ifx\svgwidth\undefined%
    \setlength{\unitlength}{104.88188976bp}%
    \ifx\svgscale\undefined%
      \relax%
    \else%
      \setlength{\unitlength}{\unitlength * \real{\svgscale}}%
    \fi%
  \else%
    \setlength{\unitlength}{\svgwidth}%
  \fi%
  \global\let\svgwidth\undefined%
  \global\let\svgscale\undefined%
  \makeatother%
  \begin{picture}(1,0.79530133)%
    \put(0,0){\includegraphics[width=\unitlength,page=1]{cross.pdf}}%
    \put(-0.02957247,0.15213876){\color[rgb]{0,0,0}\makebox(0,0)[lb]{\smash{$z_0$}}}%
    \put(0.74125724,-0.03532916){\color[rgb]{0,0,0}\makebox(0,0)[lb]{\smash{$z_1$}}}%
    \put(0.90683456,0.42306846){\color[rgb]{0,0,0}\makebox(0,0)[lb]{\smash{$z_2$}}}%
    \put(0.47644741,0.81130845){\color[rgb]{0,0,0}\makebox(0,0)[lb]{\smash{$z_3$}}}%
    \put(0.22468131,0.49590716){\color[rgb]{0,0,0}\makebox(0,0)[lb]{\smash{$z'$}}}%
    \put(0.69126518,0.63597697){\color[rgb]{0,0,0}\makebox(0,0)[lb]{\smash{$z$}}}%
    \put(0.83134653,0.19823511){\color[rgb]{0,0,0}\makebox(0,0)[lb]{\smash{$z'$}}}%
    \put(0.36362353,0.01420803){\color[rgb]{0,0,0}\makebox(0,0)[lb]{\smash{$z$}}}%
    \put(0,0){\includegraphics[width=\unitlength,page=2]{cross.pdf}}%
    \put(0.53099915,0.49805077){\color[rgb]{0,0,0}\makebox(0,0)[lb]{\smash{$z''$}}}%
    \put(0,0){\includegraphics[width=\unitlength,page=3]{cross.pdf}}%
    \put(0.37933983,0.2600991){\color[rgb]{0,0,0}\makebox(0,0)[lb]{\smash{$z''$}}}%
  \end{picture}%
\endgroup%

%% file: cross_ratio_pos.pdf_tex
\begingroup%
  \makeatletter%
  \providecommand\color[2][]{%
    \errmessage{(Inkscape) Color is used for the text in Inkscape, but the package 'color.sty' is not loaded}%
    \renewcommand\color[2][]{}%
  }%
  \providecommand\transparent[1]{%
    \errmessage{(Inkscape) Transparency is used (non-zero) for the text in Inkscape, but the package 'transparent.sty' is not loaded}%
    \renewcommand\transparent[1]{}%
  }%
  \providecommand\rotatebox[2]{#2}%
  \ifx\svgwidth\undefined%
    \setlength{\unitlength}{226.77165354bp}%
    \ifx\svgscale\undefined%
      \relax%
    \else%
      \setlength{\unitlength}{\unitlength * \real{\svgscale}}%
    \fi%
  \else%
    \setlength{\unitlength}{\svgwidth}%
  \fi%
  \global\let\svgwidth\undefined%
  \global\let\svgscale\undefined%
  \makeatother%
  \begin{picture}(1,0.71314392)%
    \put(0,0){\includegraphics[width=\unitlength,page=1]{cross_ratio_pos.pdf}}%
    \put(0.3472621,0.61518344){\color[rgb]{0,0,0}\makebox(0,0)[lb]{\smash{$\dfrac{w_m}{m_\beta w_j}$}}}%
    \put(0,0){\includegraphics[width=\unitlength,page=2]{cross_ratio_pos.pdf}}%
    \put(0.95715768,0.67155985){\color[rgb]{0,0,0}\makebox(0,0)[lb]{\smash{$\dfrac{m_\beta w_k}{w_l}$}}}%
    \put(0.96555661,0.05624372){\color[rgb]{0,0,0}\makebox(0,0)[lb]{\smash{$\dfrac{w_k}{m_\alpha w_j}$}}}%
    \put(0,0){\includegraphics[width=\unitlength,page=3]{cross_ratio_pos.pdf}}%
    \put(0.33251547,0.08396608){\color[rgb]{0,0,0}\makebox(0,0)[lb]{\smash{$\dfrac{m_\alpha w_m}{w_l}$}}}%
    \put(0,0){\includegraphics[width=\unitlength,page=4]{cross_ratio_pos.pdf}}%
    \put(0.95351008,0.33867028){\color[rgb]{0,0,0}\makebox(0,0)[lb]{\smash{$\dfrac{w_j w_l}{w_m w_k}$}}}%
    \put(0,0){\includegraphics[width=\unitlength,page=5]{cross_ratio_pos.pdf}}%
  \end{picture}%
\endgroup%

%% file: cross_ratio_neg.pdf_tex
\begingroup%
  \makeatletter%
  \providecommand\color[2][]{%
    \errmessage{(Inkscape) Color is used for the text in Inkscape, but the package 'color.sty' is not loaded}%
    \renewcommand\color[2][]{}%
  }%
  \providecommand\transparent[1]{%
    \errmessage{(Inkscape) Transparency is used (non-zero) for the text in Inkscape, but the package 'transparent.sty' is not loaded}%
    \renewcommand\transparent[1]{}%
  }%
  \providecommand\rotatebox[2]{#2}%
  \ifx\svgwidth\undefined%
    \setlength{\unitlength}{226.77165354bp}%
    \ifx\svgscale\undefined%
      \relax%
    \else%
      \setlength{\unitlength}{\unitlength * \real{\svgscale}}%
    \fi%
  \else%
    \setlength{\unitlength}{\svgwidth}%
  \fi%
  \global\let\svgwidth\undefined%
  \global\let\svgscale\undefined%
  \makeatother%
  \begin{picture}(1,0.71312499)%
    \put(0,0){\includegraphics[width=\unitlength,page=1]{cross_ratio_neg.pdf}}%
    \put(0.95850105,0.08433804){\color[rgb]{0,0,0}\makebox(0,0)[lb]{\smash{$\dfrac{m_\beta w_l}{w_k}$}}}%
    \put(0,0){\includegraphics[width=\unitlength,page=2]{cross_ratio_neg.pdf}}%
    \put(0.32763743,0.03762469){\color[rgb]{0,0,0}\makebox(0,0)[lb]{\smash{$\dfrac{w_j}{m_\beta w_m}$}}}%
    \put(0,0){\includegraphics[width=\unitlength,page=3]{cross_ratio_neg.pdf}}%
    \put(0.33302039,0.65400094){\color[rgb]{0,0,0}\makebox(0,0)[lb]{\smash{$\dfrac{m_\alpha w_l}{w_m}$}}}%
    \put(0,0){\includegraphics[width=\unitlength,page=4]{cross_ratio_neg.pdf}}%
    \put(0.9607212,0.60999856){\color[rgb]{0,0,0}\makebox(0,0)[lb]{\smash{$\dfrac{w_j}{m_\alpha w_k}$}}}%
    \put(0,0){\includegraphics[width=\unitlength,page=5]{cross_ratio_neg.pdf}}%
    \put(0.95001804,0.36005097){\color[rgb]{0,0,0}\makebox(0,0)[lb]{\smash{$\dfrac{w_m w_k}{w_j w_l}$}}}%
    \put(0,0){\includegraphics[width=\unitlength,page=6]{cross_ratio_neg.pdf}}%
    \put(-0.5836195,2.09491649){\color[rgb]{0,0,0}\makebox(0,0)[lt]{\begin{minipage}{0.10583333\unitlength}\raggedright \end{minipage}}}%
  \end{picture}%
\endgroup%

%% file: local_diagram.pdf_tex
\begingroup%
  \makeatletter%
  \providecommand\color[2][]{%
    \errmessage{(Inkscape) Color is used for the text in Inkscape, but the package 'color.sty' is not loaded}%
    \renewcommand\color[2][]{}%
  }%
  \providecommand\transparent[1]{%
    \errmessage{(Inkscape) Transparency is used (non-zero) for the text in Inkscape, but the package 'transparent.sty' is not loaded}%
    \renewcommand\transparent[1]{}%
  }%
  \providecommand\rotatebox[2]{#2}%
  \ifx\svgwidth\undefined%
    \setlength{\unitlength}{326.84979296bp}%
    \ifx\svgscale\undefined%
      \relax%
    \else%
      \setlength{\unitlength}{\unitlength * \real{\svgscale}}%
    \fi%
  \else%
    \setlength{\unitlength}{\svgwidth}%
  \fi%
  \global\let\svgwidth\undefined%
  \global\let\svgscale\undefined%
  \makeatother%
  \begin{picture}(1,0.20469489)%
    \put(0,0){\includegraphics[width=\unitlength,page=1]{local_diagram.pdf}}%
    \put(0.25277429,0.13951533){\color[rgb]{0,0,0}\makebox(0,0)[lt]{\begin{minipage}{0.2229362\unitlength}\raggedright $\cdots$\end{minipage}}}%
    \put(0.09910894,0.17659252){\color[rgb]{0,0,0}\makebox(0,0)[lt]{\begin{minipage}{0.2229362\unitlength}\raggedright $w_{j_1}$\end{minipage}}}%
    \put(0.0997301,0.10582808){\color[rgb]{0,0,0}\makebox(0,0)[lt]{\begin{minipage}{0.2229362\unitlength}\raggedright $w_{j_2}$\end{minipage}}}%
    \put(0.19418263,0.17680279){\color[rgb]{0,0,0}\makebox(0,0)[lt]{\begin{minipage}{0.2229362\unitlength}\raggedright $w_{j_3}$\end{minipage}}}%
    \put(0.19480935,0.10546692){\color[rgb]{0,0,0}\makebox(0,0)[lt]{\begin{minipage}{0.2229362\unitlength}\raggedright $w_{j_4}$\end{minipage}}}%
    \put(0,0){\includegraphics[width=\unitlength,page=2]{local_diagram.pdf}}%
    \put(0.3417424,0.17815228){\color[rgb]{0,0,0}\makebox(0,0)[lt]{\begin{minipage}{0.22293619\unitlength}\raggedright $w_{j_{2m+1}}$\end{minipage}}}%
    \put(0.34252737,0.10139889){\color[rgb]{0,0,0}\makebox(0,0)[lt]{\begin{minipage}{0.22293619\unitlength}\raggedright $w_{j_{2m+2}}$\end{minipage}}}%
    \put(0.56410287,0.08738718){\color[rgb]{0,0,0}\makebox(0,0)[lb]{\smash{}}}%
    \put(0,0){\includegraphics[width=\unitlength,page=3]{local_diagram.pdf}}%
    \put(0.7923862,0.14065365){\color[rgb]{0,0,0}\makebox(0,0)[lt]{\begin{minipage}{0.2229362\unitlength}\raggedright $\cdots$\end{minipage}}}%
    \put(0.63872086,0.17773083){\color[rgb]{0,0,0}\makebox(0,0)[lt]{\begin{minipage}{0.2229362\unitlength}\raggedright $w_{j_1}$\end{minipage}}}%
    \put(0.63934203,0.1069664){\color[rgb]{0,0,0}\makebox(0,0)[lt]{\begin{minipage}{0.2229362\unitlength}\raggedright $w_{j_2}$\end{minipage}}}%
    \put(0.73379454,0.17794113){\color[rgb]{0,0,0}\makebox(0,0)[lt]{\begin{minipage}{0.2229362\unitlength}\raggedright $w_{j_3}$\end{minipage}}}%
    \put(0.73442127,0.10660525){\color[rgb]{0,0,0}\makebox(0,0)[lt]{\begin{minipage}{0.2229362\unitlength}\raggedright $w_{j_4}$\end{minipage}}}%
    \put(0,0){\includegraphics[width=\unitlength,page=4]{local_diagram.pdf}}%
    \put(0.88042852,0.17929062){\color[rgb]{0,0,0}\makebox(0,0)[lt]{\begin{minipage}{0.22293619\unitlength}\raggedright $w_{j_{2m+1}}$\end{minipage}}}%
    \put(0.8821393,0.10253723){\color[rgb]{0,0,0}\makebox(0,0)[lt]{\begin{minipage}{0.22293619\unitlength}\raggedright $w_{j_{2m+2}}$\end{minipage}}}%
    \put(0,0){\includegraphics[width=\unitlength,page=5]{local_diagram.pdf}}%
    \put(0.17119484,0.03286587){\color[rgb]{0,0,0}\makebox(0,0)[lt]{\begin{minipage}{0.22293619\unitlength}\raggedright (a) Over-arc\end{minipage}}}%
    \put(0.69743041,0.02797065){\color[rgb]{0,0,0}\makebox(0,0)[lt]{\begin{minipage}{0.22293619\unitlength}\raggedright (b) Under-arc\end{minipage}}}%
    \put(0.52908961,0.14517474){\color[rgb]{0,0,0}\makebox(0,0)[lt]{\begin{minipage}{0.2229362\unitlength}\raggedright $m_i$\end{minipage}}}%
    \put(-0.00774742,0.14169469){\color[rgb]{0,0,0}\makebox(0,0)[lt]{\begin{minipage}{0.2229362\unitlength}\raggedright $m_i$\end{minipage}}}%
  \end{picture}%
\endgroup%

%% file: eigen.pdf_tex
\begingroup%
  \makeatletter%
  \providecommand\color[2][]{%
    \errmessage{(Inkscape) Color is used for the text in Inkscape, but the package 'color.sty' is not loaded}%
    \renewcommand\color[2][]{}%
  }%
  \providecommand\transparent[1]{%
    \errmessage{(Inkscape) Transparency is used (non-zero) for the text in Inkscape, but the package 'transparent.sty' is not loaded}%
    \renewcommand\transparent[1]{}%
  }%
  \providecommand\rotatebox[2]{#2}%
  \ifx\svgwidth\undefined%
    \setlength{\unitlength}{310.44125775bp}%
    \ifx\svgscale\undefined%
      \relax%
    \else%
      \setlength{\unitlength}{\unitlength * \real{\svgscale}}%
    \fi%
  \else%
    \setlength{\unitlength}{\svgwidth}%
  \fi%
  \global\let\svgwidth\undefined%
  \global\let\svgscale\undefined%
  \makeatother%
  \begin{picture}(1,0.31367846)%
    \put(0,0){\includegraphics[width=\unitlength,page=1]{eigen.pdf}}%
    \put(0.17542904,0.1984852){\color[rgb]{0,0,0}\makebox(0,0)[lb]{\smash{$w_{j}$}}}%
    \put(0.17676324,0.11834764){\color[rgb]{0,0,0}\makebox(0,0)[lb]{\smash{$w_{k}$}}}%
    \put(0.03903749,0.16643145){\color[rgb]{0,0,0}\makebox(0,0)[lb]{\smash{$m_{i}$}}}%
    \put(0,0){\includegraphics[width=\unitlength,page=2]{eigen.pdf}}%
    \put(0.5655856,0.07064777){\color[rgb]{0,0,0}\makebox(0,0)[lb]{\smash{$\dfrac{w_j}{m_i w_k}$}}}%
    \put(0,0){\includegraphics[width=\unitlength,page=3]{eigen.pdf}}%
    \put(0.41567643,0.25786828){\color[rgb]{0,0,0}\makebox(0,0)[lb]{\smash{$\dfrac{m_i w_j}{w_k}$}}}%
    \put(0,0){\includegraphics[width=\unitlength,page=4]{eigen.pdf}}%
    \put(0.30709517,0.16224284){\color[rgb]{0,0,0}\makebox(0,0)[lb]{\smash{$\rightarrow$}}}%
    \put(0,0){\includegraphics[width=\unitlength,page=5]{eigen.pdf}}%
    \put(0.7354716,0.16304476){\color[rgb]{0,0,0}\makebox(0,0)[lb]{\smash{$\rightarrow$}}}%
    \put(0,0){\includegraphics[width=\unitlength,page=6]{eigen.pdf}}%
    \put(0.12133891,0.22387087){\color[rgb]{0,0,0}\makebox(0,0)[lb]{\smash{$\mu_{i}$}}}%
    \put(0.88347987,0.22515936){\color[rgb]{0,0,0}\makebox(0,0)[lb]{\smash{$\mu_{i}$}}}%
    \put(0,0){\includegraphics[width=\unitlength,page=7]{eigen.pdf}}%
  \end{picture}%
\endgroup%

%% file: tetrahedron.pdf_tex
\begingroup%
  \makeatletter%
  \providecommand\color[2][]{%
    \errmessage{(Inkscape) Color is used for the text in Inkscape, but the package 'color.sty' is not loaded}%
    \renewcommand\color[2][]{}%
  }%
  \providecommand\transparent[1]{%
    \errmessage{(Inkscape) Transparency is used (non-zero) for the text in Inkscape, but the package 'transparent.sty' is not loaded}%
    \renewcommand\transparent[1]{}%
  }%
  \providecommand\rotatebox[2]{#2}%
  \ifx\svgwidth\undefined%
    \setlength{\unitlength}{340.15748031bp}%
    \ifx\svgscale\undefined%
      \relax%
    \else%
      \setlength{\unitlength}{\unitlength * \real{\svgscale}}%
    \fi%
  \else%
    \setlength{\unitlength}{\svgwidth}%
  \fi%
  \global\let\svgwidth\undefined%
  \global\let\svgscale\undefined%
  \makeatother%
  \begin{picture}(1,0.328722)%
    \put(0,0){\includegraphics[width=\unitlength,page=1]{tetrahedron.pdf}}%
    \put(0.18643028,0.04421006){\color[rgb]{0,0,0}\makebox(0,0)[lb]{\smash{$l_1$}}}%
    \put(0.10289119,0.16913694){\color[rgb]{0,0,0}\makebox(0,0)[lb]{\smash{$l_2$}}}%
    \put(0.36129683,0.09898908){\color[rgb]{0,0,0}\makebox(0,0)[lb]{\smash{$l_5$}}}%
    \put(0,0){\includegraphics[width=\unitlength,page=2]{tetrahedron.pdf}}%
    \put(0.28464074,0.2516929){\color[rgb]{0,0,0}\makebox(0,0)[lb]{\smash{$l_4$}}}%
    \put(0.79217855,0.09382201){\color[rgb]{0,0,0}\makebox(0,0)[lb]{\smash{$l_5$}}}%
    \put(0.18666541,0.14158537){\color[rgb]{0,0,0}\makebox(0,0)[lb]{\smash{$l_6$}}}%
    \put(0,0){\includegraphics[width=\unitlength,page=3]{tetrahedron.pdf}}%
    \put(0.87340541,0.16015365){\color[rgb]{0,0,0}\makebox(0,0)[lb]{\smash{$s_{23}$}}}%
    \put(0.92768647,0.23787821){\color[rgb]{0,0,0}\makebox(0,0)[lb]{\smash{$s_{34}$}}}%
    \put(0.81622574,0.24479711){\color[rgb]{0,0,0}\makebox(0,0)[lb]{\smash{$s_{42}$}}}%
    \put(0,0){\includegraphics[width=\unitlength,page=4]{tetrahedron.pdf}}%
    \put(0.03350537,0.02610989){\color[rgb]{0,0,0}\makebox(0,0)[lb]{\smash{$0$}}}%
    \put(0.36882357,0.00425859){\color[rgb]{0,0,0}\makebox(0,0)[lb]{\smash{$1$}}}%
    \put(0.3457405,0.1995572){\color[rgb]{0,0,0}\makebox(0,0)[lb]{\smash{$2$}}}%
    \put(0.19736071,0.29071233){\color[rgb]{0,0,0}\makebox(0,0)[lb]{\smash{$3$}}}%
    \put(0,0){\includegraphics[width=\unitlength,page=5]{tetrahedron.pdf}}%
    \put(0.27125153,0.09138798){\color[rgb]{0,0,0}\makebox(0,0)[lb]{\smash{$l_3$}}}%
    \put(0.61466825,0.02999696){\color[rgb]{0,0,0}\makebox(0,0)[lb]{\smash{$l_1$}}}%
    \put(0.62881352,0.13698359){\color[rgb]{0,0,0}\makebox(0,0)[lb]{\smash{$l_6$}}}%
    \put(0.52755329,0.16075259){\color[rgb]{0,0,0}\makebox(0,0)[lb]{\smash{$l_2$}}}%
    \put(0.71502967,0.24698919){\color[rgb]{0,0,0}\makebox(0,0)[lb]{\smash{$l_4$}}}%
    \put(0.70194635,0.08637831){\color[rgb]{0,0,0}\makebox(0,0)[lb]{\smash{$l_3$}}}%
    \put(0,0){\includegraphics[width=\unitlength,page=6]{tetrahedron.pdf}}%
  \end{picture}%
\endgroup%

%% file: octahedron.pdf_tex
\begingroup%
  \makeatletter%
  \providecommand\color[2][]{%
    \errmessage{(Inkscape) Color is used for the text in Inkscape, but the package 'color.sty' is not loaded}%
    \renewcommand\color[2][]{}%
  }%
  \providecommand\transparent[1]{%
    \errmessage{(Inkscape) Transparency is used (non-zero) for the text in Inkscape, but the package 'transparent.sty' is not loaded}%
    \renewcommand\transparent[1]{}%
  }%
  \providecommand\rotatebox[2]{#2}%
  \ifx\svgwidth\undefined%
    \setlength{\unitlength}{393.88698308bp}%
    \ifx\svgscale\undefined%
      \relax%
    \else%
      \setlength{\unitlength}{\unitlength * \real{\svgscale}}%
    \fi%
  \else%
    \setlength{\unitlength}{\svgwidth}%
  \fi%
  \global\let\svgwidth\undefined%
  \global\let\svgscale\undefined%
  \makeatother%
  \begin{picture}(1,0.45700908)%
    \put(0,0){\includegraphics[width=\unitlength,page=1]{octahedron.pdf}}%
    \put(0.09133648,0.25473652){\color[rgb]{0,0,0}\makebox(0,0)[lb]{\smash{$e_m$}}}%
    \put(0.43710567,0.25473652){\color[rgb]{0,0,0}\makebox(0,0)[lb]{\smash{$e_k$}}}%
    \put(0.26270461,0.16620576){\color[rgb]{0,0,0}\makebox(0,0)[lb]{\smash{$e_j$}}}%
    \put(0.26620378,0.3365264){\color[rgb]{0,0,0}\makebox(0,0)[lb]{\smash{$e_l$}}}%
    \put(0.27342832,0.43462272){\color[rgb]{0,0,0}\makebox(0,0)[lb]{\smash{$0$}}}%
    \put(0.27230357,0.05098613){\color[rgb]{0,0,0}\makebox(0,0)[lb]{\smash{$0$}}}%
    \put(0.09115654,0.31480923){\color[rgb]{0,0,0}\makebox(0,0)[lb]{\smash{$+2$}}}%
    \put(0.09198917,0.18126138){\color[rgb]{0,0,0}\makebox(0,0)[lb]{\smash{$-1$}}}%
    \put(0.43213817,0.18126138){\color[rgb]{0,0,0}\makebox(0,0)[lb]{\smash{$+1$}}}%
    \put(0.43289365,0.31637828){\color[rgb]{0,0,0}\makebox(0,0)[lb]{\smash{$-2$}}}%
    \put(0.13647661,0.00566028){\color[rgb]{0,0,0}\makebox(0,0)[lb]{\smash{(a) Positive crossing}}}%
    \put(0,0){\includegraphics[width=\unitlength,page=2]{octahedron.pdf}}%
    \put(0.22087631,0.33743802){\color[rgb]{0,0,0}\makebox(0,0)[lb]{\smash{$h^1$}}}%
    \put(0,0){\includegraphics[width=\unitlength,page=3]{octahedron.pdf}}%
    \put(0.31915539,0.33854944){\color[rgb]{0,0,0}\makebox(0,0)[lb]{\smash{$h^2$}}}%
    \put(0,0){\includegraphics[width=\unitlength,page=4]{octahedron.pdf}}%
    \put(0.32609891,0.37802492){\color[rgb]{0,0,0}\makebox(0,0)[lb]{\smash{$h^3$}}}%
    \put(0,0){\includegraphics[width=\unitlength,page=5]{octahedron.pdf}}%
    \put(0.21620271,0.37667191){\color[rgb]{0,0,0}\makebox(0,0)[lb]{\smash{$h^4$}}}%
    \put(0,0){\includegraphics[width=\unitlength,page=6]{octahedron.pdf}}%
    \put(0.33508206,0.11449125){\color[rgb]{0,0,0}\makebox(0,0)[lb]{\smash{$h_1$}}}%
    \put(0,0){\includegraphics[width=\unitlength,page=7]{octahedron.pdf}}%
    \put(0.20306455,0.11574056){\color[rgb]{0,0,0}\makebox(0,0)[lb]{\smash{$h_2$}}}%
    \put(0,0){\includegraphics[width=\unitlength,page=8]{octahedron.pdf}}%
    \put(0.21806352,0.15808188){\color[rgb]{0,0,0}\makebox(0,0)[lb]{\smash{$h_3$}}}%
    \put(0,0){\includegraphics[width=\unitlength,page=9]{octahedron.pdf}}%
    \put(0.31899082,0.15620432){\color[rgb]{0,0,0}\makebox(0,0)[lb]{\smash{$h_4$}}}%
    \put(0,0){\includegraphics[width=\unitlength,page=10]{octahedron.pdf}}%
    \put(0.23060927,0.28629909){\color[rgb]{0,0,0}\makebox(0,0)[lb]{\smash{$h^5$}}}%
    \put(0.36381192,0.27535142){\color[rgb]{0,0,0}\makebox(0,0)[lb]{\smash{}}}%
    \put(0,0){\includegraphics[width=\unitlength,page=11]{octahedron.pdf}}%
    \put(0.23036467,0.2136631){\color[rgb]{0,0,0}\makebox(0,0)[lb]{\smash{$h_5$}}}%
    \put(0,0){\includegraphics[width=\unitlength,page=12]{octahedron.pdf}}%
    \put(0.55035111,0.25555635){\color[rgb]{0,0,0}\makebox(0,0)[lb]{\smash{$e_m$}}}%
    \put(0.89612055,0.25555635){\color[rgb]{0,0,0}\makebox(0,0)[lb]{\smash{$e_k$}}}%
    \put(0.72171936,0.1670256){\color[rgb]{0,0,0}\makebox(0,0)[lb]{\smash{$e_j$}}}%
    \put(0.72521854,0.33734623){\color[rgb]{0,0,0}\makebox(0,0)[lb]{\smash{$e_l$}}}%
    \put(0.73535894,0.43544256){\color[rgb]{0,0,0}\makebox(0,0)[lb]{\smash{$0$}}}%
    \put(0.73538034,0.05180596){\color[rgb]{0,0,0}\makebox(0,0)[lb]{\smash{$0$}}}%
    \put(0.54692886,0.31969113){\color[rgb]{0,0,0}\makebox(0,0)[lb]{\smash{$-2$}}}%
    \put(0.54694168,0.17883892){\color[rgb]{0,0,0}\makebox(0,0)[lb]{\smash{$+1$}}}%
    \put(0.89115299,0.1885658){\color[rgb]{0,0,0}\makebox(0,0)[lb]{\smash{$-1$}}}%
    \put(0.89190847,0.31719811){\color[rgb]{0,0,0}\makebox(0,0)[lb]{\smash{$+2$}}}%
    \put(0.5954913,0.00648012){\color[rgb]{0,0,0}\makebox(0,0)[lb]{\smash{(a) Negative crossing}}}%
    \put(0,0){\includegraphics[width=\unitlength,page=13]{octahedron.pdf}}%
    \put(0.67989106,0.33825785){\color[rgb]{0,0,0}\makebox(0,0)[lb]{\smash{$h^4$}}}%
    \put(0,0){\includegraphics[width=\unitlength,page=14]{octahedron.pdf}}%
    \put(0.77817014,0.33936927){\color[rgb]{0,0,0}\makebox(0,0)[lb]{\smash{$h^1$}}}%
    \put(0,0){\includegraphics[width=\unitlength,page=15]{octahedron.pdf}}%
    \put(0.7851136,0.37479188){\color[rgb]{0,0,0}\makebox(0,0)[lb]{\smash{$h^2$}}}%
    \put(0,0){\includegraphics[width=\unitlength,page=16]{octahedron.pdf}}%
    \put(0.6752174,0.37749174){\color[rgb]{0,0,0}\makebox(0,0)[lb]{\smash{$h^3$}}}%
    \put(0,0){\includegraphics[width=\unitlength,page=17]{octahedron.pdf}}%
    \put(0.79409681,0.11531108){\color[rgb]{0,0,0}\makebox(0,0)[lb]{\smash{$h_4$}}}%
    \put(0,0){\includegraphics[width=\unitlength,page=18]{octahedron.pdf}}%
    \put(0.66207924,0.1165604){\color[rgb]{0,0,0}\makebox(0,0)[lb]{\smash{$h_1$}}}%
    \put(0,0){\includegraphics[width=\unitlength,page=19]{octahedron.pdf}}%
    \put(0.67707828,0.15890171){\color[rgb]{0,0,0}\makebox(0,0)[lb]{\smash{$h_2$}}}%
    \put(0,0){\includegraphics[width=\unitlength,page=20]{octahedron.pdf}}%
    \put(0.77800545,0.15702416){\color[rgb]{0,0,0}\makebox(0,0)[lb]{\smash{$h_3$}}}%
    \put(0,0){\includegraphics[width=\unitlength,page=21]{octahedron.pdf}}%
    \put(0.76810344,0.28226474){\color[rgb]{0,0,0}\makebox(0,0)[lb]{\smash{$h^5$}}}%
    \put(0,0){\includegraphics[width=\unitlength,page=22]{octahedron.pdf}}%
    \put(0.76741182,0.21650943){\color[rgb]{0,0,0}\makebox(0,0)[lb]{\smash{$h_5$}}}%
    \put(0,0){\includegraphics[width=\unitlength,page=23]{octahedron.pdf}}%
    \put(0.35174872,0.04538244){\color[rgb]{0,0,0}\makebox(0,0)[lb]{\smash{$m_\alpha$}}}%
    \put(0.33239434,0.44093928){\color[rgb]{0,0,0}\makebox(0,0)[lb]{\smash{$m_\beta$}}}%
    \put(0.63525525,0.0540332){\color[rgb]{0,0,0}\makebox(0,0)[lb]{\smash{$m_\beta$}}}%
    \put(0.63099483,0.44659474){\color[rgb]{0,0,0}\makebox(0,0)[lb]{\smash{$m_\alpha$}}}%
  \end{picture}%
\endgroup%

%% file: tetra.pdf_tex
\begingroup%
  \makeatletter%
  \providecommand\color[2][]{%
    \errmessage{(Inkscape) Color is used for the text in Inkscape, but the package 'color.sty' is not loaded}%
    \renewcommand\color[2][]{}%
  }%
  \providecommand\transparent[1]{%
    \errmessage{(Inkscape) Transparency is used (non-zero) for the text in Inkscape, but the package 'transparent.sty' is not loaded}%
    \renewcommand\transparent[1]{}%
  }%
  \providecommand\rotatebox[2]{#2}%
  \ifx\svgwidth\undefined%
    \setlength{\unitlength}{128.37082149bp}%
    \ifx\svgscale\undefined%
      \relax%
    \else%
      \setlength{\unitlength}{\unitlength * \real{\svgscale}}%
    \fi%
  \else%
    \setlength{\unitlength}{\svgwidth}%
  \fi%
  \global\let\svgwidth\undefined%
  \global\let\svgscale\undefined%
  \makeatother%
  \begin{picture}(1,0.7615096)%
    \put(0,0){\includegraphics[width=\unitlength,page=1]{tetra.pdf}}%
    \put(0.84186932,0.25449044){\color[rgb]{0,0,0}\makebox(0,0)[lb]{\smash{$l_5$}}}%
    \put(0,0){\includegraphics[width=\unitlength,page=2]{tetra.pdf}}%
    \put(0.39947732,0.53775294){\color[rgb]{0,0,0}\makebox(0,0)[lb]{\smash{$s_{23}$}}}%
    \put(0,0){\includegraphics[width=\unitlength,page=3]{tetra.pdf}}%
    \put(0.37150185,0.08536659){\color[rgb]{0,0,0}\makebox(0,0)[lb]{\smash{$l_1$}}}%
    \put(0.40898404,0.36886016){\color[rgb]{0,0,0}\makebox(0,0)[lb]{\smash{$l_6$}}}%
    \put(0.14909476,0.44870429){\color[rgb]{0,0,0}\makebox(0,0)[lb]{\smash{$l_2$}}}%
    \put(0.62497604,0.65929963){\color[rgb]{0,0,0}\makebox(0,0)[lb]{\smash{$l_4$}}}%
    \put(0.6027717,0.23476609){\color[rgb]{0,0,0}\makebox(0,0)[lb]{\smash{$l_3$}}}%
    \put(0.28908396,0.62434764){\color[rgb]{0,0,0}\makebox(0,0)[lb]{\smash{$v^2$}}}%
    \put(0,0){\includegraphics[width=\unitlength,page=4]{tetra.pdf}}%
    \put(0.71417798,0.59308522){\color[rgb]{0,0,0}\makebox(0,0)[lb]{\smash{$v_4$}}}%
    \put(0,0){\includegraphics[width=\unitlength,page=5]{tetra.pdf}}%
    \put(0.13108606,-0.00604129){\color[rgb]{0,0,0}\makebox(0,0)[lb]{\smash{$v_1$}}}%
    \put(0,0){\includegraphics[width=\unitlength,page=6]{tetra.pdf}}%
    \put(0.83827637,0.10789406){\color[rgb]{0,0,0}\makebox(0,0)[lb]{\smash{$v_5$}}}%
  \end{picture}%
\endgroup%

%% file: developing.pdf_tex
\begingroup%
  \makeatletter%
  \providecommand\color[2][]{%
    \errmessage{(Inkscape) Color is used for the text in Inkscape, but the package 'color.sty' is not loaded}%
    \renewcommand\color[2][]{}%
  }%
  \providecommand\transparent[1]{%
    \errmessage{(Inkscape) Transparency is used (non-zero) for the text in Inkscape, but the package 'transparent.sty' is not loaded}%
    \renewcommand\transparent[1]{}%
  }%
  \providecommand\rotatebox[2]{#2}%
  \ifx\svgwidth\undefined%
    \setlength{\unitlength}{338.31750825bp}%
    \ifx\svgscale\undefined%
      \relax%
    \else%
      \setlength{\unitlength}{\unitlength * \real{\svgscale}}%
    \fi%
  \else%
    \setlength{\unitlength}{\svgwidth}%
  \fi%
  \global\let\svgwidth\undefined%
  \global\let\svgscale\undefined%
  \makeatother%
  \begin{picture}(1,0.47974566)%
    \put(0,0){\includegraphics[width=\unitlength,page=1]{developing.pdf}}%
    \put(0.22504025,0.4289286){\color[rgb]{0,0,0}\makebox(0,0)[lb]{\smash{$p$}}}%
    \put(0,0){\includegraphics[width=\unitlength,page=2]{developing.pdf}}%
    \put(0.2301859,0.03029349){\color[rgb]{0,0,0}\makebox(0,0)[lb]{\smash{$q$}}}%
    \put(0,0){\includegraphics[width=\unitlength,page=3]{developing.pdf}}%
    \put(0.06892592,0.15995122){\color[rgb]{0,0,0}\makebox(0,0)[lb]{\smash{$L$}}}%
    \put(0,0){\includegraphics[width=\unitlength,page=4]{developing.pdf}}%
    \put(0.85328475,0.38872989){\color[rgb]{0,0,0}\makebox(0,0)[lb]{\smash{$\mathbb{R}$}}}%
    \put(0.10428193,0.42829002){\color[rgb]{0,0,0}\makebox(0,0)[lb]{\smash{$\nu(p)$}}}%
    \put(0.10495883,0.04324953){\color[rgb]{0,0,0}\makebox(0,0)[lb]{\smash{$\nu(q)$}}}%
    \put(0,0){\includegraphics[width=\unitlength,page=5]{developing.pdf}}%
    \put(0.13352316,0.28899116){\color[rgb]{0,0,0}\makebox(0,0)[lb]{\smash{$e_j$}}}%
    \put(0.27333875,0.31420554){\color[rgb]{0,0,0}\makebox(0,0)[lb]{\smash{$e_k$}}}%
    \put(0,0){\includegraphics[width=\unitlength,page=6]{developing.pdf}}%
    \put(0.6950551,0.24506356){\color[rgb]{0,0,0}\makebox(0,0)[lb]{\smash{$\widetilde{e}_k$}}}%
    \put(0.52934303,0.269736){\color[rgb]{0,0,0}\makebox(0,0)[lb]{\smash{$\widetilde{e}_j$}}}%
    \put(0,0){\includegraphics[width=\unitlength,page=7]{developing.pdf}}%
    \put(0.72227952,0.14197379){\color[rgb]{0,0,0}\makebox(0,0)[lb]{\smash{$v^0_k$}}}%
    \put(0.64928842,0.3255322){\color[rgb]{0,0,0}\makebox(0,0)[lb]{\smash{$v^1_k$}}}%
    \put(0,0){\includegraphics[width=\unitlength,page=8]{developing.pdf}}%
    \put(0.56766117,0.35027125){\color[rgb]{0,0,0}\makebox(0,0)[lb]{\smash{$v^1_j$}}}%
    \put(0,0){\includegraphics[width=\unitlength,page=9]{developing.pdf}}%
    \put(0.53471648,0.15340248){\color[rgb]{0,0,0}\makebox(0,0)[lb]{\smash{$v^0_j$}}}%
    \put(0,0){\includegraphics[width=\unitlength,page=10]{developing.pdf}}%
    \put(0.17569138,0.29561083){\color[rgb]{0,0,0}\makebox(0,0)[lb]{\smash{$x$}}}%
    \put(0,0){\includegraphics[width=\unitlength,page=11]{developing.pdf}}%
    \put(0.76309262,0.35781425){\color[rgb]{0,0,0}\makebox(0,0)[lb]{\smash{$\widetilde{x}$}}}%
    \put(0.83753737,0.17830404){\color[rgb]{0,0,0}\makebox(0,0)[lb]{\smash{$\widetilde{y}$}}}%
    \put(0,0){\includegraphics[width=\unitlength,page=12]{developing.pdf}}%
    \put(0.20261719,0.16978251){\color[rgb]{0,0,0}\makebox(0,0)[lb]{\smash{$y$}}}%
    \put(0,0){\includegraphics[width=\unitlength,page=13]{developing.pdf}}%
    \put(0.8488833,0.33546965){\color[rgb]{0,0,0}\makebox(0,0)[lb]{\smash{$v_x$}}}%
    \put(0.8936757,0.26743941){\color[rgb]{0,0,0}\makebox(0,0)[lb]{\smash{$v_y$}}}%
    \put(0,0){\includegraphics[width=\unitlength,page=14]{developing.pdf}}%
    \put(0.36448098,0.36790432){\color[rgb]{0,0,0}\makebox(0,0)[lb]{\smash{$g$}}}%
    \put(0,0){\includegraphics[width=\unitlength,page=15]{developing.pdf}}%
    \put(0.92495316,0.34913446){\color[rgb]{0,0,0}\makebox(0,0)[lb]{\smash{$g \cdot v_x$}}}%
  \end{picture}%
\endgroup%

%% file: developing2.pdf_tex
\begingroup%
  \makeatletter%
  \providecommand\color[2][]{%
    \errmessage{(Inkscape) Color is used for the text in Inkscape, but the package 'color.sty' is not loaded}%
    \renewcommand\color[2][]{}%
  }%
  \providecommand\transparent[1]{%
    \errmessage{(Inkscape) Transparency is used (non-zero) for the text in Inkscape, but the package 'transparent.sty' is not loaded}%
    \renewcommand\transparent[1]{}%
  }%
  \providecommand\rotatebox[2]{#2}%
  \ifx\svgwidth\undefined%
    \setlength{\unitlength}{338.31750825bp}%
    \ifx\svgscale\undefined%
      \relax%
    \else%
      \setlength{\unitlength}{\unitlength * \real{\svgscale}}%
    \fi%
  \else%
    \setlength{\unitlength}{\svgwidth}%
  \fi%
  \global\let\svgwidth\undefined%
  \global\let\svgscale\undefined%
  \makeatother%
  \begin{picture}(1,0.47974566)%
    \put(0,0){\includegraphics[width=\unitlength,page=1]{developing2.pdf}}%
    \put(0.24868668,0.4289286){\color[rgb]{0,0,0}\makebox(0,0)[lb]{\smash{$p$}}}%
    \put(0,0){\includegraphics[width=\unitlength,page=2]{developing2.pdf}}%
    \put(0.25383233,0.03029349){\color[rgb]{0,0,0}\makebox(0,0)[lb]{\smash{$q$}}}%
    \put(0,0){\includegraphics[width=\unitlength,page=3]{developing2.pdf}}%
    \put(0.09939362,0.16336187){\color[rgb]{0,0,0}\makebox(0,0)[lb]{\smash{$L$}}}%
    \put(0,0){\includegraphics[width=\unitlength,page=4]{developing2.pdf}}%
    \put(0.12792836,0.42829002){\color[rgb]{0,0,0}\makebox(0,0)[lb]{\smash{$\nu(p)$}}}%
    \put(0.12860526,0.04324953){\color[rgb]{0,0,0}\makebox(0,0)[lb]{\smash{$\nu(q)$}}}%
    \put(0,0){\includegraphics[width=\unitlength,page=5]{developing2.pdf}}%
    \put(0.15716959,0.28899116){\color[rgb]{0,0,0}\makebox(0,0)[lb]{\smash{$e_j$}}}%
    \put(0.30103238,0.30938519){\color[rgb]{0,0,0}\makebox(0,0)[lb]{\smash{$e_k$}}}%
    \put(0,0){\includegraphics[width=\unitlength,page=6]{developing2.pdf}}%
    \put(0.55298946,0.269736){\color[rgb]{0,0,0}\makebox(0,0)[lb]{\smash{$\widetilde{e}_j$}}}%
    \put(0,0){\includegraphics[width=\unitlength,page=7]{developing2.pdf}}%
    \put(0.78958206,0.14806732){\color[rgb]{0,0,0}\makebox(0,0)[lb]{\smash{$v^0_k$}}}%
    \put(0.71045183,0.35213515){\color[rgb]{0,0,0}\makebox(0,0)[lb]{\smash{$v^1_k$}}}%
    \put(0,0){\includegraphics[width=\unitlength,page=8]{developing2.pdf}}%
    \put(0.59130759,0.35027125){\color[rgb]{0,0,0}\makebox(0,0)[lb]{\smash{$v^1_j$}}}%
    \put(0,0){\includegraphics[width=\unitlength,page=9]{developing2.pdf}}%
    \put(0.55836291,0.15340248){\color[rgb]{0,0,0}\makebox(0,0)[lb]{\smash{$v^0_j$}}}%
    \put(0.65168482,0.10892854){\color[rgb]{0,0,0}\makebox(0,0)[lb]{\smash{$\widetilde{z}$}}}%
    \put(0,0){\includegraphics[width=\unitlength,page=10]{developing2.pdf}}%
    \put(0.75509016,0.26831511){\color[rgb]{0,0,0}\makebox(0,0)[lb]{\smash{$\widetilde{e}_k$}}}%
    \put(0.23137872,0.16785561){\color[rgb]{0,0,0}\makebox(0,0)[lb]{\smash{$z$}}}%
    \put(0,0){\includegraphics[width=\unitlength,page=11]{developing2.pdf}}%
    \put(0.37940702,0.35973235){\color[rgb]{0,0,0}\makebox(0,0)[lb]{\smash{$g$}}}%
    \put(0,0){\includegraphics[width=\unitlength,page=12]{developing2.pdf}}%
  \end{picture}%
\endgroup%

%% file: figure_eight_diagram.pdf_tex
\begingroup%
  \makeatletter%
  \providecommand\color[2][]{%
    \errmessage{(Inkscape) Color is used for the text in Inkscape, but the package 'color.sty' is not loaded}%
    \renewcommand\color[2][]{}%
  }%
  \providecommand\transparent[1]{%
    \errmessage{(Inkscape) Transparency is used (non-zero) for the text in Inkscape, but the package 'transparent.sty' is not loaded}%
    \renewcommand\transparent[1]{}%
  }%
  \providecommand\rotatebox[2]{#2}%
  \ifx\svgwidth\undefined%
    \setlength{\unitlength}{170.79652801bp}%
    \ifx\svgscale\undefined%
      \relax%
    \else%
      \setlength{\unitlength}{\unitlength * \real{\svgscale}}%
    \fi%
  \else%
    \setlength{\unitlength}{\svgwidth}%
  \fi%
  \global\let\svgwidth\undefined%
  \global\let\svgscale\undefined%
  \makeatother%
  \begin{picture}(1,0.74009954)%
    \put(0,0){\includegraphics[width=\unitlength,page=1]{figure_eight_diagram.pdf}}%
    \put(0.49550042,0.4984129){\color[rgb]{0,0,0}\makebox(0,0)[lb]{\smash{$w_1$}}}%
    \put(0.70467553,0.46116068){\color[rgb]{0,0,0}\makebox(0,0)[lb]{\smash{$w_2$}}}%
    \put(0.50495677,0.25564182){\color[rgb]{0,0,0}\makebox(0,0)[lb]{\smash{$w_3$}}}%
    \put(0.51674873,0.10989775){\color[rgb]{0,0,0}\makebox(0,0)[lb]{\smash{$w_4$}}}%
    \put(0.05195491,0.38219044){\color[rgb]{0,0,0}\makebox(0,0)[lb]{\smash{$w_5$}}}%
    \put(0.30195777,0.37860258){\color[rgb]{0,0,0}\makebox(0,0)[lb]{\smash{$w_6$}}}%
    \put(0.11592032,0.61343607){\color[rgb]{0,0,0}\makebox(0,0)[lb]{\smash{$g_1$}}}%
    \put(0.34635317,-0.01585869){\color[rgb]{0,0,0}\makebox(0,0)[lb]{\smash{$g_2$}}}%
    \put(0.68341492,0.35399991){\color[rgb]{0,0,0}\makebox(0,0)[lb]{\smash{$g_3$}}}%
    \put(0.82761579,0.70583561){\color[rgb]{0,0,0}\makebox(0,0)[lb]{\smash{$g_4$}}}%
    \put(0,0){\includegraphics[width=\unitlength,page=2]{figure_eight_diagram.pdf}}%
  \end{picture}%
\endgroup%

%% file: rule.pdf_tex
\begingroup%
  \makeatletter%
  \providecommand\color[2][]{%
    \errmessage{(Inkscape) Color is used for the text in Inkscape, but the package 'color.sty' is not loaded}%
    \renewcommand\color[2][]{}%
  }%
  \providecommand\transparent[1]{%
    \errmessage{(Inkscape) Transparency is used (non-zero) for the text in Inkscape, but the package 'transparent.sty' is not loaded}%
    \renewcommand\transparent[1]{}%
  }%
  \providecommand\rotatebox[2]{#2}%
  \ifx\svgwidth\undefined%
    \setlength{\unitlength}{114.31097076bp}%
    \ifx\svgscale\undefined%
      \relax%
    \else%
      \setlength{\unitlength}{\unitlength * \real{\svgscale}}%
    \fi%
  \else%
    \setlength{\unitlength}{\svgwidth}%
  \fi%
  \global\let\svgwidth\undefined%
  \global\let\svgscale\undefined%
  \makeatother%
  \begin{picture}(1,0.42156038)%
    \put(0.2988789,0.31773587){\color[rgb]{0,0,0}\makebox(0,0)[lt]{\begin{minipage}{0.29993373\unitlength}\raggedright $V_i$\end{minipage}}}%
    \put(0,0){\includegraphics[width=\unitlength,page=1]{rule.pdf}}%
    \put(0.72158393,0.23805952){\color[rgb]{0,0,0}\makebox(0,0)[lt]{\begin{minipage}{0.29993373\unitlength}\raggedright $g_k$\end{minipage}}}%
    \put(0.463266,0.1377548){\color[rgb]{0,0,0}\makebox(0,0)[lt]{\begin{minipage}{0.29993373\unitlength}\raggedright $V_j$\end{minipage}}}%
    \put(0.04719378,0.05656433){\color[rgb]{0,0,0}\makebox(0,0)[lt]{\begin{minipage}{0.29993373\unitlength}\raggedright $L$\end{minipage}}}%
  \end{picture}%
\endgroup%

%% file: white.pdf_tex
\begingroup%
  \makeatletter%
  \providecommand\color[2][]{%
    \errmessage{(Inkscape) Color is used for the text in Inkscape, but the package 'color.sty' is not loaded}%
    \renewcommand\color[2][]{}%
  }%
  \providecommand\transparent[1]{%
    \errmessage{(Inkscape) Transparency is used (non-zero) for the text in Inkscape, but the package 'transparent.sty' is not loaded}%
    \renewcommand\transparent[1]{}%
  }%
  \providecommand\rotatebox[2]{#2}%
  \ifx\svgwidth\undefined%
    \setlength{\unitlength}{226.77165354bp}%
    \ifx\svgscale\undefined%
      \relax%
    \else%
      \setlength{\unitlength}{\unitlength * \real{\svgscale}}%
    \fi%
  \else%
    \setlength{\unitlength}{\svgwidth}%
  \fi%
  \global\let\svgwidth\undefined%
  \global\let\svgscale\undefined%
  \makeatother%
  \begin{picture}(1,0.53241725)%
    \put(0,0){\includegraphics[width=\unitlength,page=1]{white.pdf}}%
    \put(0.26053854,0.5079224){\color[rgb]{0,0,0}\makebox(0,0)[lb]{\smash{$w_1$}}}%
    \put(0.14388351,0.25097067){\color[rgb]{0,0,0}\makebox(0,0)[lb]{\smash{$w_2$}}}%
    \put(0.3326982,0.26509449){\color[rgb]{0,0,0}\makebox(0,0)[lb]{\smash{$w_3$}}}%
    \put(0.47496597,0.37805595){\color[rgb]{0,0,0}\makebox(0,0)[lb]{\smash{$w_4$}}}%
    \put(0.62381774,0.26605309){\color[rgb]{0,0,0}\makebox(0,0)[lb]{\smash{$w_5$}}}%
    \put(0.79241229,0.28817878){\color[rgb]{0,0,0}\makebox(0,0)[lb]{\smash{$w_6$}}}%
    \put(0.48435477,0.14451183){\color[rgb]{0,0,0}\makebox(0,0)[lb]{\smash{$w_7$}}}%
    \put(0,0){\includegraphics[width=\unitlength,page=2]{white.pdf}}%
    \put(0.61780218,0.55165627){\color[rgb]{0,0,0}\makebox(0,0)[lb]{\smash{$g_1$}}}%
    \put(0.08553194,0.46273878){\color[rgb]{0,0,0}\makebox(0,0)[lb]{\smash{$g_2$}}}%
    \put(0.40365684,0.43625525){\color[rgb]{0,0,0}\makebox(0,0)[lb]{\smash{$g_4$}}}%
    \put(0,0){\includegraphics[width=\unitlength,page=3]{white.pdf}}%
    \put(0.41305566,0.00498437){\color[rgb]{0,0,0}\makebox(0,0)[lb]{\smash{$g_3$}}}%
    \put(0.97086465,0.38085036){\color[rgb]{0,0,0}\makebox(0,0)[lb]{\smash{$g_5$}}}%
  \end{picture}%
\endgroup%

%% file: potentialfunction.bbl
\begin{thebibliography}{10}

\bibitem{carter2001geometric}
J.~S. Carter, S.~Kamada, and M.~Saito.
\newblock {Geometric interpretations of quandle homology}.
\newblock {\em Journal of knot theory and its ramifications}, 10(03):345--386,
  2001.

\bibitem{cho2016optimistic}
J.~Cho.
\newblock {Optimistic limit of the colored Jones polynomial and the existence
  of a solution}.
\newblock {\em Proceedings of the American Mathematical Society},
  144(4):1803--1814, 2016.

\bibitem{cho_2016}
J.~Cho.
\newblock {Optimistic limits of the colored Jones polynomials and the complex
  volumes of hyperbolic linkes}.
\newblock {\em Journal of the Australian Mathematical Society},
  100(3):303–337, 2016.

\bibitem{cho2013optimistic}
J.~Cho and J.~Murakami.
\newblock Optimistic limits of the colored jones polynomials.
\newblock {\em J. Korean Math. Soc}, 50(3):641--693, 2013.

\bibitem{CYZ2018hikami}
J.~Cho, S.~Yoon, and C.~K.~Zickert.
\newblock {On the Hikami-Inoue conjecture}.
\newblock {\em arXiv preprint arXiv:1801.08288}, 2018.

\bibitem{garoufalidis2015complex}
S.~Garoufalidis, D.~P. Thurston, and C.~Zickert.
\newblock {The complex volume of $\mathrm{SL}(n,\mathbb{C})$-representations of
  3-manifolds}.
\newblock {\em Duke Mathematical Journal}, 164(11):2099--2160, 2015.

\bibitem{kim2016octahedral}
H.~Kim, S.~Kim, and S.~Yoon.
\newblock {Octahedral developing of knot complement I: pseudo-hyperbolic
  structure}.
\newblock {\em arXiv preprint arXiv:1612.02928}, 2016.

\bibitem{neumann2004extended}
W.~D. Neumann.
\newblock {Extended Bloch group and the Cheeger--Chern--Simons class}.
\newblock {\em Geometry \& Topology}, 8(1):413--474, 2004.

\bibitem{thurston1999hyperbolic}
D.~Thurston.
\newblock {Hyperbolic volume and the Jones polynomial}.
\newblock {\em handwritten note in Grenoble summer school}, 1999.

\bibitem{weeks2005computation}
J.~Weeks.
\newblock Computation of hyperbolic structures in knot theory.
\newblock In {\em Handbook of knot theory}, pages 461--480. Elsevier, 2005.

\bibitem{yokota2002potential}
Y.~Yokota.
\newblock {On the potential functions for the hyperbolic structures of a knot
  complement}.
\newblock {\em Geometry \& Topology Monographs}, 4:303--311, 2002.

\bibitem{yoon2018volume}
S.~Yoon.
\newblock {The volume and Chern-Simons invariant of a Dehn-filled manifold}.
\newblock {\em arXiv preprint arXiv:1801.08288}, 2018.

\bibitem{zagier2007dilogarithm}
D.~Zagier.
\newblock The dilogarithm function.
\newblock In {\em {Frontiers in Number theory, Physics, and Geometry II}},
  pages 3--65. Springer, 2007.

\end{thebibliography}
